\documentclass[11pt,reqno]{amsart}
\usepackage{amssymb}
\usepackage{amsmath}
\usepackage{amsfonts}

\newtheorem{theorem}{Theorem}
\theoremstyle{plain}

\newtheorem{example}{Example}

\newtheorem{lemma}{Lemma}

\newtheorem{proposition}{Proposition}
\newtheorem{remark}{Remark}

\numberwithin{equation}{section}
\numberwithin{equation}{section}
\input{tcilatex}

\begin{document}
\title[Stabilization of coupled wave]{Behaviors of the energy of solutions
of two coupled wave equations with nonlinear damping on a compact manifold
with boundary.}
\author{M. Daoulatli}
\address{University of Dammam, King Saudi Arabia \& University of Carthage,
Tunisia}
\email[M. Daoulatli]{moez.daoulatli@infcom.rnu.tn}
\date{\today }

\begin{abstract}
In this paper we study the behaviors of the the energy of solutions of
coupled wave equations on a compact manifold with boundary in the case of \
indirect nonlinear damping . Only one of the two equations is directly
damped by a localized nonlinear damping term. Under geometric conditions on
both the coupling and the damping regions we prove that the rate of decay of
the energy of \ smooth solutions of the system is determined from a first
order differential equation .
\end{abstract}

\maketitle

\section{Introduction and Statement of the results}

Let $\left( \Omega ,g_{0}\right) $ be a $C^{\infty }$ compact connected
n-dimensional Riemannian manifold with boundary $\Gamma .$ We denote by $%
\Delta $ the Laplace-Beltrami operator on $\Omega $ for the metric $g_{0}.$
We consider a system of coupled wave equations with nonlinear damping%
\begin{equation}
\left\{ 
\begin{array}{ll}
\partial _{t}^{2}u-\Delta u+b(x)v+a(x)g(\partial _{t}u)=0 & \text{in }%
\mathbb{R}_{+}^{\ast }\times \Omega  \\ 
\partial _{t}^{2}v-\Delta v+b(x)u=0 & \text{in }\mathbb{R}_{+}^{\ast }\times
\Omega  \\ 
u=v=0 & \text{on }\mathbb{R}_{+}^{\ast }\times \Gamma  \\ 
\left( u\left( 0,x\right) ,\partial _{t}u\left( 0,x\right) \right) =\left(
u_{0},u_{1}\right) \text{ and }\left( v\left( 0,x\right) ,\partial
_{t}v\left( 0,x\right) \right) =\left( v_{0},v_{1}\right)  & \text{in }%
\Omega ,%
\end{array}%
\right.   \label{system intro}
\end{equation}%
where $g:%
\mathbb{R}
\longrightarrow 
\mathbb{R}
$ is a continuous, monotone increasing function, $g(0)=0$. In addition we
assume that%
\begin{equation*}
\begin{array}{l}
g\left( y\right) y\leq M_{0}y^{2},~~\left\vert y\right\vert <1 \\ 
my^{2}\leq g\left( y\right) y\leq My^{2},~~\left\vert y\right\vert \geq 1 \\ 
\left\Vert g^{\prime }\right\Vert _{L^{\infty }}\leq M_{1},%
\end{array}%
\end{equation*}%
for some positive real numbers $M_{0},~m,~M~$\ and $M_{1}$. In this paper,
we deal with real solutions, the general case can be treated in the same
way. With the system above we associate the energy functional given by%
\begin{equation}
\begin{array}{l}
E_{u,v}\left( t\right) =\frac{1}{2}\dint_{\Omega }\left\vert \nabla u\left(
t,x\right) \right\vert ^{2}+\left\vert \nabla v\left( t,x\right) \right\vert
^{2}+\left\vert \partial _{t}u\left( t,x\right) \right\vert ^{2}+\left\vert
\partial _{t}v\left( t,x\right) \right\vert ^{2}dx \\ 
\text{ \ \ \ \ \ \ \ \ \ \ }+\dint_{\Omega }b\left( x\right) u\left(
t,x\right) v\left( t,x\right) dx.%
\end{array}
\label{energy intro}
\end{equation}

We assume that $a$ and $b$ are two nonnegative smooth functions such that 
\begin{equation}
\left\Vert b\right\Vert _{\infty }\leq \frac{1-\delta }{\lambda ^{2}},
\label{b condition}
\end{equation}%
for some $\delta >0,$ where $\lambda $\ is the Poincar\'{e}'s constant on $%
\Omega .$ Under these assumptions we have%
\begin{equation}
\begin{array}{l}
E_{u,v}\left( 0\right) =\frac{1}{2}\dint_{\Omega }\left\vert \nabla
u_{0}\left( x\right) \right\vert ^{2}+\left\vert \nabla v_{0}\left( x\right)
\right\vert ^{2}+\left\vert u_{1}\left( x\right) \right\vert ^{2}+\left\vert
v_{1}\left( x\right) \right\vert ^{2}dx+\dint_{\Omega }b\left( x\right)
u_{0}\left( x\right) v_{0}\left( x\right) dx \\ 
\text{ \ \ \ \ \ \ \ \ \ \ }\geq \frac{\delta }{2}\dint_{\Omega }\left\vert
\nabla u_{0}\left( x\right) \right\vert ^{2}+\left\vert \nabla v_{0}\left(
x\right) \right\vert ^{2}+\left\vert u_{1}\left( x\right) \right\vert
^{2}+\left\vert v_{1}\left( x\right) \right\vert ^{2}dx,%
\end{array}
\label{b assumption intro}
\end{equation}%
for all $\left( u_{0},v_{0},u_{1},v_{1}\right) \in \mathcal{H}%
=~H_{0}^{1}\left( \Omega \right) \times H_{0}^{1}\left( \Omega \right)
\times L^{2}\left( \Omega \right) \times L^{2}\left( \Omega \right) .$

The nonlinear evolution equation $\left( \ref{system intro}\right) $\ can be
rewritten under the form 
\begin{equation}
\left\{ 
\begin{array}{c}
\dfrac{d}{dt}U+\mathcal{A}U+\mathcal{B}U=0 \\ 
U\left( 0\right) =U_{0}\in \mathcal{H}%
\end{array}%
\right.  \label{abstract system}
\end{equation}%
where 
\begin{equation*}
\begin{array}{c}
U=\left( 
\begin{array}{c}
u \\ 
v \\ 
\partial _{t}u \\ 
\partial _{t}v%
\end{array}%
\right) ,U_{0}=\left( 
\begin{array}{c}
u_{0} \\ 
v_{0} \\ 
u_{1} \\ 
v_{1}%
\end{array}%
\right) ,%
\end{array}%
\end{equation*}%
and the unbounded operator $\mathcal{A}$ on $\mathcal{H}$ is defined by%
\begin{equation*}
\begin{array}{c}
\mathcal{A=}\left( 
\begin{array}{cccc}
0 & 0 & -Id & 0 \\ 
0 & 0 & 0 & -Id \\ 
-\Delta & b & 0 & 0 \\ 
b & -\Delta & 0 & 0%
\end{array}%
\right)%
\end{array}%
\end{equation*}%
with domain 
\begin{equation*}
\begin{array}{l}
D\left( \mathcal{A}\right) =\left\{ U\in \mathcal{H};\mathcal{A}U\in 
\mathcal{H}\right\} \\ 
=\left( H_{0}^{1}\left( \Omega \right) \cap H^{2}\left( \Omega \right)
\right) \times \left( H_{0}^{1}\left( \Omega \right) \cap H^{2}\left( \Omega
\right) \right) \times H_{0}^{1}\left( \Omega \right) \times H_{0}^{1}\left(
\Omega \right) ,%
\end{array}%
\end{equation*}%
and%
\begin{equation*}
\begin{array}{c}
\mathcal{B}U=\left( 
\begin{array}{c}
0 \\ 
0 \\ 
a\left( x\right) g\left( \partial _{t}u\right) \\ 
0%
\end{array}%
\right)%
\end{array}%
\end{equation*}%
Under our assumptions and from the nonlinear semi-group theory (see for
example \cite{Ala nonl}), we can infer that for $U_{0}\in \mathcal{H},$ the
problem $\left( \ref{abstract system}\right) $ admits a unique solution $%
U\in C^{0}\left( 
\mathbb{R}
_{+},\mathcal{H}\right) .$ Moreover we have the following energy estimate%
\begin{equation}
\begin{array}{c}
E_{u,v}\left( t\right) -E_{u,v}\left( 0\right) =-\dint_{0}^{t}\dint_{\Omega
}a\left( x\right) g\left( \partial _{t}u\left( s,x\right) \right) \partial
_{t}u\left( s,x\right) dxds%
\end{array}
\label{energy estimate intro}
\end{equation}%
for all $t\geq 0.$ In addition, since $g^{\prime }\in L^{\infty }\left( 
\mathbb{R}
\right) ,$ then if $U_{0}\in D\left( \mathcal{A}\right) ,$ we have $U\in
C\left( 
\mathbb{R}
_{+},D\left( \mathcal{A}\right) \right) $ and%
\begin{equation}
\begin{array}{c}
E_{\partial _{t}u,\partial _{t}v}\left( t\right) -E_{\partial _{t}u,\partial
_{t}v}\left( 0\right) =-\dint_{0}^{t}\dint_{\Omega }a\left( x\right)
g^{\prime }\left( \partial _{t}u\left( s,x\right) \right) \left\vert
\partial _{t}^{2}u\left( s,x\right) \right\vert ^{2}dxds.%
\end{array}
\label{hight energy estimate intro}
\end{equation}

The systems like $\left( \ref{system intro}\right) $ appear in many physical
situations. Indirect damping of reversible systems occurs in several
applications in engineering and mechanics. In general it is impossible or
too expansive to damp all the components of the state, so it is important to
study stabilization properties of coupled systems with a reduced number of
feedbacks.\newline

The case of a linear damping and constant coupling b in $\left( \ref{system
intro}\right) $ has already been treated in \cite{alab-canar-komo}. They
showed that the System $\left( \ref{system intro}\right) $ cannot be
exponentially stable and that the energy decays polynomially. In \cite%
{alab-lea stabilization} Alabau et al generalized these results to cases for
which the coupling $b=b(x)$ and the damping term $a=a(x)$ satisfy the
Piecewise Multipliers Geometric Condition (PMGC) \cite{alab-lea
stabilization}. This geometric assumption is a generalization of the usual
multiplier geometric condition (or $\Gamma $-condition) of \cite{Zua, lios}
and is much more restrictive than the sharp Geometric Control Condition
(GCC). In \cite{alab-lea control} Alabau et al generalized this result and
they proved that the system $\left( \ref{system intro}\right) $ is
polynomially stable when the regions $\{a>0\}$ and $\{b>0\}$ both satisfy
the Geometric Control Condition and the coupling term satisfies a smallness
assumption. \ This result has generalized by Aloui et al \cite{Aloui-daou},
by assuming a more natural smallness condition on the infinity norm of the
coupling term $b.$ Finally we quote the result of Fu \cite{X fu} in which he
shows the logarithmic decay property without any geometric conditions on the
effective damping domain.

The problem of the indirect nonlinear damping has been studied by Alabau et
Al \cite{Ala nonl} when the system is coupled by the velocity. \ In this
case they show that the energy of these kinds of system decays as fast as
that of the corresponding scalar nonlinearly damped equation. Hence, the
coupling through velocities allows a full transmission of the damping
effects. To our knowledge no results seems to be known in the case of
indirect nonlinear damping for a coupled system coupled in displacements.

The goal of this paper is to determine the rate of decay of the energy of
coupled wave system with indirect nonlinear damping and coupled in
displacements. More precisely, we prove, under some geometric conditions on
the localized damping domain and the localized coupling domain, that the
rate of decay of the energy is determined from a first order differential
equation . In addition, we obtain that if the behavior of the damping is
close to the linear case, then the linear and the nonlinear case has the
same rate of decay. In the other case we find that the rate of decay of the
coupled system is close to the one obtained for a single damped wave
equation.

The optimality of our results is a open questions. Lower energy estimates
have been established in \cite{alabau1, Ala nonl} and \cite{ala lower 1} for
scalar one-dimensional wave equations, scalar Petrowsky equations in
two-dimensions and one-dimensional wave systems coupled by velocities. These
results can be extended to the case of one-dimensional wave systems coupled
by displacement. In our case we obtain a quasi-optimal energy decay formula
when the behavior of the damping is not close to the linear one.

A natural necessary and sufficient condition to obtain controllability for
wave equations is to assume that the control set satisfies the Geometric
Control Condition (GCC) defined in \cite{BLR,RT}. For a subset $\omega $ of $%
\Omega $ and $T>0$, we shall say that $\left( \omega ,T\right) $ satisfies
GCC if every geodesic traveling at speed one in $\Omega $ meets $\omega $ in
a time $t<T$. We say that $\omega $ satisfies GCC if there exists $T>0$ such
that $\left( \omega ,T\right) $ satisfies GCC. We also set $T_{\omega
}=inf\left\{ T>0;(\omega ,T)\text{ satisfies }GCC\right\} .$

We denote by $\omega =\left\{ a\left( x\right) >0\right\} $ the control set
and by $\mathcal{O}=\left\{ b\left( x\right) >0\right\} $ the coupling set.

\begin{description}
\item[Assumption $\left( A1\right) $: ] Unique continuation property:
\end{description}

There exists $T_{0}>0,$ such that the only solution of the system%
\begin{equation}
\left\{ 
\begin{array}{ll}
\partial _{t}^{2}u_{1}-\Delta u_{1}+b(x)u_{2}=0 & \text{in }\left(
0,T_{0}\right) \times \Omega \\ 
\partial _{t}^{2}u_{2}-\Delta u_{2}+b(x)u_{1}=0 & \text{in }\left(
0,T_{0}\right) \times \Omega \\ 
u_{1}=u_{2}=0 & \text{on }\left( 0,T_{0}\right) \times \Gamma \\ 
a\left( x\right) u_{1}=0 & \text{on }\left( 0,T_{0}\right) \times \Omega \\ 
u_{1}\in H^{1}\left( \left( 0,T_{0}\right) \times \Omega \right) \text{ and }%
u_{2}\in L^{2}\left( \left( 0,T_{0}\right) \times \Omega \right) , & 
\end{array}%
\right.  \label{unique continuation system}
\end{equation}%
is the null one $u_{1}=u_{2}=0$.

Note that the unique continuation assumption above is valid if we assume
that $\omega \cap \mathcal{O}$ satisfies the GCC (see \cite{Aloui-daou}).
Also according to Alabau et al. \cite[Proposition 4.7]{alab-lea control} we
have the following result$:$ We assume that $\omega $ and $\mathcal{O}$
satisfy the GCC, then if $\left\Vert b\right\Vert _{\infty }\leq \min \left( 
\frac{1}{5\lambda ^{2}},\frac{1}{50\lambda \sqrt{C_{T_{\omega }}}}\right) ,$
there exists $T^{\ast }\geq \max \left( T_{\omega },T_{\mathcal{O}}\right) $
such that if $T_{0}>T^{\ast }$ then the only solution of the system $\left( %
\ref{unique continuation system}\right) $ is the null one.\newline

In order to characterize decay rates for the energy, we need to introduce
several special functions, which in turn will depend on the growth of $g$
near the origin. According to \cite{las-tat} there exists a concave
continuous, strictly increasing function $h_{0}$, linear at infinity with $%
h_{0}(0)=0$ such that%
\begin{equation}
h_{0}\left( g\left( s\right) s\right) \geq \epsilon _{0}\left( \left\vert
s\right\vert ^{2}+\left\vert g\left( s\right) \right\vert ^{2}\right)
,~\left\vert s\right\vert \leq 1,  \label{h0 defin}
\end{equation}%
where $\epsilon _{0}$ is a positive constant$.$ We set%
\begin{equation}
\begin{array}{c}
h\left( s\right) =m_{a}\left( \Omega \right) h_{0}\left( \frac{s}{%
m_{a}\left( \Omega \right) }\right) ,\text{ where }m_{a}=a\left( x\right) dx.%
\end{array}
\label{h defin}
\end{equation}

\begin{description}
\item[Assumption $\left( A2\right) $] We assume that there exists $%
0<r_{0}\leq 1$ such that the function $h^{-1}\in C^{3}\left(
(0,r_{0}]\right) $ and strictly convex. In addition we suppose that 
\begin{equation}
\begin{array}{c}
\underset{s\rightarrow 0}{\lim }h^{-1}\left( s\right) =\underset{%
s\rightarrow 0}{\lim }\left( h^{-1}\right) ^{\prime }\left( s\right) =%
\underset{s\rightarrow 0}{\lim }s\left( h^{-1}\right) ^{\prime \prime
}\left( s\right) =\underset{s\rightarrow 0}{\lim }s^{2}\left( h^{-1}\right)
^{\prime \prime \prime }\left( s\right) =0,%
\end{array}
\label{assumption1}
\end{equation}%
and there exist $\beta >1$ and $\alpha _{0}>0,$ such that 
\begin{equation}
\begin{array}{l}
\underset{s\rightarrow \infty }{\lim }s\left( h^{-1}\right) ^{\prime }\left(
1/s^{\beta }\right) =\alpha _{0}, \\ 
\left( h^{-1}\right) ^{\prime }\left( s\right) \leq \beta s\left(
h^{-1}\right) ^{\prime \prime }\left( s\right) ,\text{ for all }s\in \left[
0,r_{0}\right] , \\ 
\left( \beta ^{2}-\beta \right) s\left( h^{-1}\right) ^{\prime \prime
}\left( s\right) +\beta ^{2}s^{2}\left( h^{-1}\right) ^{\prime \prime \prime
}\left( s\right) \geq 0,\text{ for all }s\in \left[ 0,r_{0}\right] .%
\end{array}
\label{assumption2}
\end{equation}%
Moreover, we assume that if $\beta s\left( h^{-1}\right) ^{\prime \prime
}\left( s\right) -\left( h^{-1}\right) ^{\prime }\left( s\right) >0,$ for
all $s\in (0,r_{0}],$ then there exists $\alpha _{1}>0,$ such that%
\begin{equation}
\begin{array}{c}
\frac{\left( h^{-1}\right) ^{\prime }\left( s\right) \left( \left( \beta
^{2}-\beta \right) s\left( h^{-1}\right) ^{\prime \prime }\left( s\right)
+\beta ^{2}s^{2}\left( h^{-1}\right) ^{\prime \prime \prime }\left( s\right)
\right) }{\beta s\left( h^{-1}\right) ^{\prime \prime }\left( s\right)
-\left( h^{-1}\right) ^{\prime }\left( s\right) }\leq \alpha _{1},\text{ for
all }s\in \left[ 0,r_{0}\right] .%
\end{array}
\label{assumption3}
\end{equation}
\end{description}

We know that in the case of linear damping we have%
\begin{equation*}
\begin{array}{c}
\medskip E_{u,v}\left( t\right) \leq \frac{c}{t}\sum_{i=0}^{1}E_{\partial
_{t}^{i}u,\partial _{t}^{i}v}\left( 0\right) ,~t>0, \\ 
\medskip U_{0}=\left( u_{0},v_{0},u_{1},v_{1}\right) \in \left(
H_{0}^{1}\left( \Omega \right) \cap H^{2}\left( \Omega \right) \right)
^{2}\times \left( H_{0}^{1}\left( \Omega \right) \right) ^{2},%
\end{array}%
\end{equation*}%
so we cannot expect to obtain a better rate of decay in the case of
nonlinear damping. More precisely we have the following result.

\begin{theorem}
We suppose that $a$ and $b$ are two smooth non-negative functions and the
conditions (\ref{b condition}) and the assumption A2 hold. In addition, we
assume that $\omega $ and $\mathcal{O}$ satisfy the GCC and the assumption
A1 holds$.$ The solution $U\left( t\right) =\left( u\left( t\right) ,v\left(
t\right) ,\partial _{t}u\left( t\right) ,\partial _{t}v\left( t\right)
\right) $ of the system $\left( \ref{system intro}\right) $\ then satisfies%
\begin{equation}
\begin{array}{l}
\medskip E_{u,v}\left( t\right) \leq C\left( 1+\sum_{i=0}^{1}E_{\partial
_{t}^{i}u,\partial _{t}^{i}v}\left( 0\right) \right) \left( \varphi \left(
t\right) \right) ^{-1},~t>0, \\ 
\medskip U_{0}=\left( u_{0},v_{0},u_{1},v_{1}\right) \in \left(
H_{0}^{1}\left( \Omega \right) \cap H^{2}\left( \Omega \right) \right)
^{2}\times \left( H_{0}^{1}\left( \Omega \right) \right) ^{2},%
\end{array}
\label{energy decay rate}
\end{equation}%
where $C$ is positive constant and $\varphi $ is a solution of the following
ODE%
\begin{equation}
\begin{array}{l}
\dfrac{d\varphi }{dt}-\frac{\epsilon _{0}}{2C_{1}}\varphi \left(
h^{-1}\right) ^{\prime }\left( 1/\varphi ^{\beta }\right) =0,~0<\varphi
^{-\beta }\left( 0\right) \leq r_{0}~\text{such that} \\ 
\left( h^{-1}\right) ^{\prime }\left( \varphi ^{-\beta }\left( 0\right)
\right) <\inf \left( \frac{2C_{1}\delta }{\epsilon _{0}}\left(
8C_{T}+2\lambda ^{2}+1\right) ^{-1},\frac{1}{\epsilon _{0}}\left( \frac{1}{m}%
+M^{2}\right) ^{-1}\right) ,%
\end{array}
\label{ode theorem 1}
\end{equation}%
where $C_{1}=C\left( T,\left\Vert b\right\Vert _{\infty },\left\Vert
a\right\Vert _{\infty }\right) $. Moreover, 
\begin{equation}
\begin{array}{c}
\text{for }U_{0}\in \mathcal{H},\underset{t\rightarrow \infty }{\lim }%
\medskip E_{u,v}\left( t\right) =0.%
\end{array}
\label{limit energy space}
\end{equation}
\end{theorem}

\begin{remark}
The smallness condition on the infinity norm of $b$ is required to ensure
that the only solution of the system 
\begin{equation}
\begin{array}{c}
\left\{ 
\begin{array}{l}
-\Delta u+b(x)v=0 \\ 
-\Delta v+b(x)u=0 \\ 
\begin{array}{ll}
u=v=0 & \text{on }\Gamma%
\end{array}%
\end{array}%
\right.%
\end{array}
\label{counterexample-utelda}
\end{equation}%
is the null one.
\end{remark}

\begin{remark}
We note that $\theta \left( t\right) =\dfrac{1}{\varphi \left( t\right) }$
is a solution of the following ODE%
\begin{equation}
\begin{array}{c}
\dfrac{d\theta }{dt}+\frac{\epsilon _{0}}{2C_{1}}\theta \left( h^{-1}\right)
^{\prime }\left( \theta ^{\beta }\right) =0.%
\end{array}
\label{one over phi ode}
\end{equation}
\end{remark}

To prove our result it is sufficient to show the integrability of $\varphi
^{\prime }E_{u,v}$ on $(0,\infty )$. For this purpose we show an estimate on
a functional $X(t)$ which is equivalent to the weighted energy functional
(See \cite{daou ext} for similar idea). Also we prove a weighted
observability estimate for the wave equation with a potential. In addition,
we use the unique continuation hypotheses (A1) to prove a weak observability
estimate of the weighted $L^{2}$-norm of the solution.

\subsection{Some examples of decay rates and lower energy estimates}

We give some examples of feedback growths together with the resulting energy
decay rate when applying our results. For clarity of exposition we will deal
with the damping which satisfies strict bounds, i.e. by saying $g(s)s\simeq
f(s)$ we will mean there are constants $m,M$ so that $mf(s)\leq g(s)s\leq
Mf(s)$. In the sequel $\tilde{C}$ denotes a generic positive constant which
is independents of the energy of the initial data and setting $%
E_{0}=1+\sum_{i=0}^{1}E_{\partial _{t}^{i}u,\partial _{t}^{i}v}\left(
0\right) $. Below we assume that $\varphi \left( 0\right) $ verifies the
condition $\left( \ref{ode theorem 1}\right) .$

First we give an explicit upper bound of solutions of the ordinary
differential equation $\left( \ref{one over phi ode}\right) .$

\begin{lemma}
We assume that there exists $0<r_{0}\leq 1$ such that the function $%
h^{-1}\in C^{2}\left( (0,r_{0}]\right) ,$ monotone increasing and strictly
convex. In addition we suppose that there exists $\alpha >0,$ such that 
\begin{equation}
\left( h^{-1}\right) ^{\prime }\left( s\right) \leq \alpha s\left(
h^{-1}\right) ^{\prime \prime }\left( s\right) ,\text{ for all }s\in
(0,r_{0}].  \label{gamma assumption}
\end{equation}%
Let $\theta $ be a solution of the following ODE%
\begin{equation*}
\dfrac{d\theta }{dt}+C\theta \left( h^{-1}\right) ^{\prime }\left( \theta
^{\beta }\right) =0,~\text{such that }0<\theta \left( 0\right) \leq r_{0},
\end{equation*}%
where $C$ is a positive constant and $\beta \geq 1$. We have%
\begin{equation*}
\begin{array}{c}
\theta \left( t\right) \leq \left( \left( \left( h^{-1}\right) ^{\prime
}\right) ^{-1}\left( \frac{\alpha /\beta C}{t+k_{0}}\right) \right)
^{1/\beta },\text{ for all }t\geq 0,%
\end{array}%
\end{equation*}%
where%
\begin{equation}
\frac{\alpha /\beta C}{\left( h^{-1}\right) ^{\prime }\left( r_{0}\right) }%
\leq k_{0}\leq \frac{\alpha /\beta C}{\left( h^{-1}\right) ^{\prime }\left(
\theta ^{\beta }\left( 0\right) \right) }.  \label{k assumption}
\end{equation}
\end{lemma}

\begin{proof}
Let 
\begin{equation*}
\psi \left( t\right) =\left( \left( \left( h^{-1}\right) ^{\prime }\right)
^{-1}\left( \frac{\alpha /\beta C}{t+k_{0}}\right) \right) ^{1/\beta },\text{
for all }t\geq 0.
\end{equation*}%
\ Direct computations and $\left( \ref{assumption2}\right) $, give%
\begin{equation*}
\psi ^{\prime }\left( t\right) \geq -\frac{\alpha }{\beta }\frac{\psi \left(
t\right) }{t+k_{0}},\text{ for all }t\geq 0.
\end{equation*}%
On the other hand, it is easy to see that%
\begin{equation*}
\psi \left( t\right) \left( h^{-1}\right) ^{\prime }\left( \psi ^{\beta
}\left( t\right) \right) =\frac{\alpha }{\beta C}\frac{\psi \left( t\right) 
}{t+k_{0}},\text{ for all }t\geq 0.
\end{equation*}%
Therefore, using $\left( \ref{k assumption}\right) $, we conclude that 
\begin{equation*}
\dfrac{d\psi }{dt}+C\psi \left( h^{-1}\right) ^{\prime }\left( \psi ^{\beta
}\right) \geq 0,\text{ }\psi \left( 0\right) \geq \theta \left( 0\right) .
\end{equation*}%
The desired result follows from \cite[Lemma 1]{daou jmaa}.
\end{proof}

Now we give some examples.

\begin{example}[Linearly bounded case]
Suppose $g\left( s\right) s\simeq s^{2}$. According to (\ref{h0 defin}),
auxiliary function $h_{0}$ which may be defined as $h_{0}(y)=(cy)^{\gamma }$
with $1/2<\gamma <1$ and for suitable constant $c>0$. We use the ODE 
\begin{equation*}
\begin{array}{l}
\dfrac{d\varphi }{dt}-\frac{\epsilon _{0}}{2C_{1}}\varphi \left(
h_{0}^{-1}\right) ^{\prime }\left( \frac{1}{m_{a}\left( \Omega \right) }%
\varphi ^{-\frac{\gamma }{1-\gamma }}\right) =0.%
\end{array}%
\end{equation*}%
Consequently, 
\begin{equation*}
\begin{array}{l}
E_{u,v}(t)\leq \frac{\tilde{C}E_{0}}{t+1},\qquad t\geq 0.%
\end{array}%
\end{equation*}
\end{example}

\begin{example}
Suppose $g\left( s\right) s\simeq s^{2}\left( \ln \left( 1/s\right) \right)
^{-p}$, $0<\left\vert s\right\vert <1/2$ for some $p>0$. According to (\ref%
{h0 defin}), auxiliary function $h_{0}$ which may be defined as $%
h_{0}(y)=(cy)^{\gamma }$ with $1/2<\gamma <1$ and for suitable constant $c>0$%
. We use the ODE 
\begin{equation*}
\begin{array}{l}
\dfrac{d\varphi }{dt}-\frac{\epsilon _{0}}{2C_{1}}\varphi \left(
h_{0}^{-1}\right) ^{\prime }\left( \frac{1}{m_{a}\left( \Omega \right) }%
\varphi ^{-\frac{\gamma }{1-\gamma }}\right) =0.%
\end{array}%
\end{equation*}%
Consequently, 
\begin{equation*}
\begin{array}{l}
E_{u,v}(t)\leq \frac{\tilde{C}E_{0}}{t+1},\qquad t\geq 0.%
\end{array}%
\end{equation*}
\end{example}

\begin{example}[The Polynomial Case]
Suppose $g\left( s\right) s\simeq |s|^{p+1}$, $0<\left\vert s\right\vert <1$
for some $p>1$. According to (\ref{h0 defin}), auxiliary function $h_{0}$
which may be defined as $h_{0}(y)=(cy)^{2/(p+1)}$ for suitable constant $c>0$
(determined by the coefficients in the polynomial bound on the damping $g(s)$%
). We use the ODE 
\begin{equation*}
\begin{array}{l}
\dfrac{d\varphi }{dt}-\frac{\epsilon _{0}}{2C_{1}}\varphi \left(
h_{0}^{-1}\right) ^{\prime }\left( \frac{1}{m_{a}\left( \Omega \right) }%
\varphi ^{-\beta }\right) =0.%
\end{array}%
\end{equation*}%
Consequently, 
\begin{equation*}
\begin{array}{l}
E_{u,v}(t)\leq \frac{\tilde{C}E_{0}}{\left( t+1\right) ^{2/\beta \left(
p-1\right) }},\qquad t\geq 0,%
\end{array}%
\end{equation*}%
where%
\begin{equation*}
\begin{array}{l}
\begin{array}{ll}
\beta >1 & \text{if }p\geq 3 \\ 
\beta =\frac{2}{P-1} & \text{if }p<3.%
\end{array}%
\end{array}%
\end{equation*}
\end{example}

\begin{example}[Exponential damping at the origin]
\label{example:exp at origin}Assume: $g\left( s\right) =s^{3}e^{-1/s^{2}}$, $%
0<\left\vert s\right\vert <1$. First we need to determine $h_{0}$ according
to (\ref{h0 defin}). Setting $h_{0}(g(y)y)=cy^{2}$, we see that 
\begin{equation*}
h_{0}^{-1}(y)=\sqrt{y/c}\cdot g\left( \sqrt{y/c}\right) =c^{-2}y^{2}\exp
(-c/y)
\end{equation*}%
We use the ODE 
\begin{equation*}
\begin{array}{l}
\dfrac{d\varphi }{dt}-\frac{\epsilon _{0}}{2C_{1}}\varphi \left(
h_{0}^{-1}\right) ^{\prime }\left( \frac{1}{m_{a}\left( \Omega \right) }%
\varphi ^{-\beta }\right) =0,%
\end{array}%
\end{equation*}%
to obtain 
\begin{equation*}
\begin{array}{l}
E_{u,v}(t)\leq \frac{\tilde{C}E_{0}}{\left( \ln \left( t+2\right) \right)
^{1/\beta }},\qquad t\geq 0,%
\end{array}%
\end{equation*}%
for all $\beta >1.$
\end{example}

\begin{example}[Exponential damping at the origin]
\label{example:exp at origin copy(1)}Assume: $g\left( s\right)
=s^{3}e^{-e^{1/s^{2}}}$, $0<\left\vert s\right\vert <1$. First we need to
determine $h_{0}$ according to (\ref{h0 defin}). Setting $%
h_{0}(g(y)y)=cy^{2} $, we see that 
\begin{equation*}
h_{0}^{-1}(y)=\sqrt{y/c}\cdot g\left( \sqrt{y/c}\right) =c^{-2}y^{2}\exp
(-\exp \left( c/y\right) )
\end{equation*}%
We use the ODE 
\begin{equation*}
\begin{array}{l}
\dfrac{d\varphi }{dt}-\frac{\epsilon _{0}}{2C_{1}}\varphi \left(
h_{0}^{-1}\right) ^{\prime }\left( \frac{1}{m_{a}\left( \Omega \right) }%
\varphi ^{-\beta }\right) =0,%
\end{array}%
\end{equation*}%
to obtain 
\begin{equation*}
\begin{array}{l}
E_{u,v}(t)\leq \frac{\tilde{C}E_{0}}{\left( \ln \ln \left( t+e^{2}\right)
\right) ^{1/\beta }},\qquad t\geq 0,%
\end{array}%
\end{equation*}%
for all $\beta >1.$
\end{example}

We finish this part by giving a result on the lower estimate of the energy
of the one-dimensional coupled wave system.

\begin{proposition}
We suppose that $\Omega =(0,1)$ and $g$ is a odd function. We set 
\begin{equation*}
\begin{array}{c}
h^{-1}\left( s\right) =g\left( \sqrt{s}\right) \sqrt{s},\text{ for }s\geq 0.%
\end{array}%
\end{equation*}%
We assume that $a$ and $b$ are two smooth non-negative functions and the
conditions (\ref{b condition}) and the assumption A2 hold. In addition, we
suppose that $\omega $ and $\mathcal{O}$ satisfy the GCC and the assumption
A1 holds$.$ Let $U\left( t\right) =\left( u\left( t\right) ,v\left( t\right)
,\partial _{t}u\left( t\right) ,\partial _{t}v\left( t\right) \right) $ be
the solution of the system $\left( \ref{system intro}\right) ,$\ then there
exists $T_{0}>0$ such that%
\begin{equation*}
\begin{array}{c}
E_{u,v}\left( t\right) \geq \left( \frac{\psi \left( t\right) }{4\sqrt{%
E_{\partial _{t}u,\partial _{t}v}\left( 0\right) }}\right) ^{2},~t\geq T_{0},
\\ 
U_{0}=\left( u_{0},v_{0},u_{1},v_{1}\right) \in \left( H_{0}^{1}\left(
\Omega \right) \cap H^{2}\left( \Omega \right) \right) ^{2}\times \left(
H_{0}^{1}\left( \Omega \right) \right) ^{2},%
\end{array}%
\end{equation*}%
where $\psi $ is a solution of the following ODE%
\begin{equation}
\dfrac{d\psi }{dt}+\left\Vert a\right\Vert _{L^{\infty }}\psi \left(
h^{-1}\right) ^{\prime }\left( \psi \right) =0,~0<\psi \left( 0\right) \leq 4%
\sqrt{E_{\partial _{t}u,\partial _{t}v}\left( 0\right) E_{u,v}\left(
T_{0}\right) }.  \label{psi lower ode}
\end{equation}
\end{proposition}

\begin{proof}
We proceed as in \cite{Ala nonl} and using the fact that%
\begin{equation*}
h^{-1}\left( s\right) \leq s\left( h^{-1}\right) ^{\prime }\left( s\right) ,
\end{equation*}%
we see that there exists $T_{0}>0$ such that%
\begin{equation*}
\dfrac{d\sqrt{E_{u,v}}}{dt}+\left\Vert a\right\Vert _{L^{\infty }}\sqrt{%
E_{u,v}}\left( h^{-1}\right) ^{\prime }\left( 4\sqrt{E_{\partial
_{t}u,\partial _{t}v}\left( 0\right) }\sqrt{E_{u,v}}\right) \geq 0,\text{ }%
t\geq T_{0}.
\end{equation*}%
Since $\psi $ is a solution of $\left( \ref{psi lower ode}\right) ,$ then
using \cite[Lemma 1]{daou jmaa} we conclude that 
\begin{equation*}
\sqrt{E_{u,v}\left( t\right) }\geq \frac{\psi \left( t\right) }{4\sqrt{%
E_{\partial _{t}u,\partial _{t}v}\left( 0\right) }},\text{ }t\geq T_{0}.
\end{equation*}
\end{proof}

\section{Proof of Theorem 1}

First we give the following weighted observability estimate for the wave
equation with potential.

\begin{proposition}
\label{observability lemma} Let $\gamma ,\delta >0$ and $\chi \in L^{\infty
}(\Omega )$ satisfying

\begin{itemize}
\item $\chi \geq 0$ or else,

\item $\left\Vert \chi \right\Vert _{\infty }\leq \frac{\gamma ^{2}-\delta }{%
\lambda ^{2}}$.
\end{itemize}

Let $\phi $ be a positive function in $C^{2}\left( 
\mathbb{R}
_{+}\right) $ such that%
\begin{equation}
\begin{array}{l}
\text{for every }I\subset \subset 
\mathbb{R}
_{+},\text{ there exist }m,M>0 \\ 
m\leq \phi \left( s\right) \leq M,\text{ for all }s\in I\text{.}%
\end{array}
\label{PHY assumption}
\end{equation}%
In addition, if $\phi ^{\prime }$ is not the null function, we assume that$\ 
$there exists a positive constant $K$ such that%
\begin{equation}
\underset{%
\mathbb{R}
_{+}}{\sup }\left\vert \frac{\phi ^{\prime \prime }\left( t\right) }{\phi
^{\prime }\left( t\right) }\right\vert \leq K.  \label{PHY assumption 1}
\end{equation}%
Moreover we suppose that the function%
\begin{equation}
\begin{array}{c}
\text{ \ \ }t\longmapsto \left\vert \frac{\phi ^{\prime }\left( t\right) }{%
\phi \left( t\right) }\right\vert \text{ is decreasing and }\underset{%
t\rightarrow +\infty }{\lim }\left\vert \frac{\phi ^{\prime }\left( t\right) 
}{\phi \left( t\right) }\right\vert =0.\text{\ \ }%
\end{array}
\label{PHY assumption 2}
\end{equation}%
We consider also $\psi $ nonnegative smooth function on $\Omega $ such that
the set $\left( \mathcal{V}:=\left\{ \psi \left( x\right) >0\right\}
,T\right) $ satisfies the GCC. Then there exists $C_{T}>0,$ such that for
all $\left( u_{0},u_{1}\right) \in H_{0}^{1}\left( \Omega \right) \times
L^{2}\left( \Omega \right) ,$ $f\in L_{loc}^{2}\left( 
\mathbb{R}
_{+},L^{2}\left( \Omega \right) \right) $ and all $t>0,$ the solution of 
\begin{equation}
\left\{ 
\begin{array}{ll}
\partial _{t}^{2}u-\gamma ^{2}\Delta u+\chi \left( x\right) u=f & \text{in }%
\mathbb{R}_{+}^{\ast }\times \Omega \\ 
u=0 & \text{on }\mathbb{R}_{+}^{\ast }\times \Gamma \\ 
\left( u\left( 0,x\right) ,\partial _{t}u\left( 0,x\right) \right) =\left(
u_{0},u_{1}\right) & \text{in }\Omega%
\end{array}%
\right.  \label{System onservability lemma}
\end{equation}%
satisfies with%
\begin{equation*}
E_{u}\left( t\right) =\frac{1}{2}\int_{\Omega }\gamma ^{2}\left\vert \nabla
u\left( t,x\right) \right\vert ^{2}+\left\vert \partial _{t}u\left(
t,x\right) \right\vert ^{2}+\chi \left( x\right) \left\vert u\left(
t,x\right) \right\vert ^{2}dx,
\end{equation*}%
the inequality 
\begin{equation}
\int_{t}^{t+T}\phi \left( s\right) E_{u}\left( s\right) ds\leq C_{T}\left(
\int_{t}^{t+T}\phi \left( s\right) \int_{\Omega }\left( \psi \left( x\right)
\left\vert \partial _{t}u\left( s,x\right) \right\vert ^{2}+\left\vert
f\left( s,x\right) \right\vert ^{2}\right) dxds\right) .
\label{observability by a}
\end{equation}
\end{proposition}

\begin{proof}
First we remark that 
\begin{equation*}
E_{u}\left( t\right) \simeq \frac{1}{2}\int_{\Omega }\left\vert \nabla
u\left( t,x\right) \right\vert ^{2}+\left\vert \partial _{t}u\left(
t,x\right) \right\vert ^{2}dx.
\end{equation*}%
To prove the estimate $\left( \ref{observability by a}\right) ,$ we argue by
contradiction. We assume that there exist a positive sequence $\left(
t_{n}\right) ,~$a sequence of functions $f_{n}\in L_{loc}^{2}\left( 
\mathbb{R}
_{+},L^{2}\left( \Omega \right) \right) $ and a sequence $\left( u_{n}\left(
t\right) \right) $ of solutions of the system $\left( \ref{System
onservability lemma}\right) $ with initial data $\left(
u_{0,n},u_{1,n}\right) \in H_{0}^{1}\left( \Omega \right) \times L^{2}\left(
\Omega \right) ,$ such that%
\begin{equation}
\begin{array}{l}
\int_{t_{n}}^{t_{n}+T}\phi \left( s\right) E_{u_{n}}\left( s\right) ds \\ 
\geq n\left( \dint_{t_{n}}^{t_{n}+T}\phi \left( s\right) \int_{\Omega }\psi
\left( x\right) \left\vert \partial _{t}u_{n}\left( s,x\right) \right\vert
^{2}+\left\vert f_{n}\left( s,x\right) \right\vert ^{2}dxds\right) .%
\end{array}
\label{proof lemma 1 observa bility wave contradiction}
\end{equation}

\begin{description}
\item[1$^{st}$ case] $t_{n}\underset{n\rightarrow +\infty }{\longrightarrow }%
\infty .$ Setting%
\begin{equation*}
\begin{array}{c}
\lambda _{n}^{2}=\int_{t_{n}}^{t_{n}+T}\phi \left( s\right) E_{u_{n}}\left(
s\right) ds\text{ and }v_{n}\left( t\right) =\frac{1}{\lambda _{n}}\left(
\phi \left( s\right) \right) ^{1/2}u_{n}\left( t+t_{n}\right) ,%
\end{array}%
\end{equation*}%
Therefore $v_{n}$ is a solution of the following system 
\begin{equation*}
\left\{ 
\begin{array}{ll}
\partial _{t}^{2}v_{n}-\gamma ^{2}\Delta v_{n}+\chi \left( x\right) v_{n}=%
\frac{1}{\lambda _{n}}\left( \phi \left( t_{n}+t\right) \right)
^{1/2}f_{n}\left( t_{n}+t,x\right) +f_{n}^{1}\left( t,x\right) & \text{in }%
\mathbb{R}_{+}^{\ast }\times \Omega \\ 
v_{n}=0 & \text{on }\mathbb{R}_{+}^{\ast }\times \Gamma \\ 
\left( v_{n}\left( 0,x\right) ,\partial _{t}v_{n}\left( 0,x\right) \right)
=\left( v_{n,0},v_{n,1}\right) & \text{in }\Omega%
\end{array}%
\right.
\end{equation*}%
where%
\begin{equation*}
\begin{array}{l}
f_{n}^{1}\left( t,x\right) =\frac{1}{2}\phi ^{\prime \prime }\left(
t+t_{n}\right) \left( \phi \left( t+t_{n}\right) \right) ^{-1}v_{n}-\frac{1}{%
4}\left( \phi ^{\prime }\left( t+t_{n}\right) \right) ^{2}\phi \left(
t+t_{n}\right) ^{-2}v_{n} \\ 
\text{ \ \ \ \ \ \ \ \ \ \ \ }+\frac{1}{\lambda _{n}}\phi ^{\prime }\left(
t+t_{n}\right) \left( \phi \left( t+t_{n}\right) \right) ^{-1/2}\partial
_{t}u_{n}.%
\end{array}%
\end{equation*}%
Thanks to $\left( \ref{proof lemma 1 observa bility wave contradiction}%
\right) ,$\ we get%
\begin{equation}
\begin{array}{l}
\dint_{0}^{T}\dint_{\Omega }\left\vert \nabla v_{n}\left( s,x\right)
\right\vert ^{2}+\frac{\phi \left( s+t_{n}\right) }{\lambda _{n}^{2}}%
\left\vert \partial _{t}u_{n}\left( s+t_{n},x\right) \right\vert ^{2}dxds=1,
\\ 
\frac{1}{\lambda _{n}^{2}}\dint_{0}^{T}\dint_{\Omega }\phi \left(
s+t_{n}\right) \left( \psi \left( x\right) \left\vert \partial
_{t}u_{n}\left( t_{n}+s,x\right) \right\vert ^{2}+\left\vert f_{n}\left(
t_{n}+s,x\right) \right\vert ^{2}\right) dxds\leq \frac{1}{n}\text{ }.%
\end{array}
\label{lemm cotrad}
\end{equation}%
Using Poincare's inequality we deduce that%
\begin{equation*}
\begin{array}{c}
\dint_{0}^{T}\dint_{\Omega }\left\vert v_{n}\left( s,x\right) \right\vert
^{2}dxds\leq \lambda ^{2}.%
\end{array}%
\end{equation*}%
Utilizing the first part of $\left( \ref{lemm cotrad}\right) $ and the
estimate above, we deduce that there exists $~\alpha _{1}>0,$ such that 
\begin{equation*}
\begin{array}{c}
\dint_{0}^{T}E_{v_{n}}\left( s\right) ds\leq \alpha _{1}.%
\end{array}%
\end{equation*}%
\ A combination of the first part of $\left( \ref{lemm cotrad}\right) ,$ $%
\left( \ref{PHY assumption 1}\right) $ and $\left( \ref{PHY assumption 2}%
\right) ,$ gives 
\begin{equation*}
\begin{array}{c}
\underset{n\rightarrow \infty }{\lim }\dint_{0}^{T}\dint_{\Omega }\left\vert
f_{n}^{1}\left( t_{n}+s,x\right) \right\vert ^{2}dxds=0\text{.}%
\end{array}%
\end{equation*}%
It is easy to see that%
\begin{equation*}
\begin{array}{l}
\dint_{0}^{T}\dint_{\Omega }\psi \left( x\right) \left\vert \partial
_{t}v_{n}\left( s,x\right) \right\vert ^{2}dxds \\ 
\leq \dint_{0}^{T}\dint_{\Omega }\frac{1}{\lambda _{n}^{2}}\phi \left(
t_{n}+s\right) \psi \left( x\right) \left\vert \partial _{t}u_{n}\left(
t_{n}+s,x\right) \right\vert ^{2}+\left\vert \frac{\phi ^{\prime }\left(
t_{n}\right) }{\phi \left( t_{n}\right) }\right\vert ^{2}\psi \left(
x\right) \left\vert v_{n}\left( s,x\right) \right\vert ^{2}dxds \\ 
\leq \dint_{0}^{T}\dint_{\Omega }\frac{1}{\lambda _{n}^{2}}\phi \left(
t_{n}+s\right) \psi \left( x\right) \left\vert \partial _{t}u_{n}\left(
t_{n}+s,x\right) \right\vert ^{2}dxds+\alpha _{2}\left\vert \frac{\phi
^{\prime }\left( t_{n}\right) }{\phi \left( t_{n}\right) }\right\vert ^{2}%
\underset{n\rightarrow +\infty }{\longrightarrow }0,%
\end{array}%
\end{equation*}%
noting that in the last result we have used the second part of $\left( \ref%
{lemm cotrad}\right) $ and $\left( \ref{PHY assumption 2}\right) $.

On the other hand, According to \cite{Aloui-daou}, we know that%
\begin{equation*}
\begin{array}{l}
\dint_{0}^{T}\dint_{\Omega }\left\vert \nabla v_{n}\left( s,x\right)
\right\vert ^{2}+\left\vert \partial _{t}v_{n}\left( s,x\right) \right\vert
^{2}dxds \\ 
\leq C_{T}\dint_{0}^{T}\dint_{\Omega }\psi \left( x\right) \left\vert
\partial _{t}v_{n}\left( s,x\right) \right\vert ^{2}+\frac{1}{\lambda
_{n}^{2}}\phi \left( t_{n}+s\right) \left\vert f_{n}\left( t_{n}+s,x\right)
\right\vert ^{2}+\left\vert f_{n}^{1}\left( s,x\right) \right\vert ^{2}dxds%
\end{array}%
\end{equation*}%
and the contradiction follows from the fact that the RHS of the estimate
above goes to zero as n goes to infinity and%
\begin{equation*}
\begin{array}{c}
1=\dint_{0}^{T}\dint_{\Omega }\left\vert \nabla v_{n}\left( s,x\right)
\right\vert ^{2}+\frac{\phi \left( s+t_{n}\right) }{\lambda _{n}^{2}}%
\left\vert \partial _{t}u_{n}\left( s+t_{n},x\right) \right\vert ^{2}dxds \\ 
\leq 2\dint_{0}^{T}\dint_{\Omega }\left\vert \nabla v_{n}\left( s,x\right)
\right\vert ^{2}+\left\vert \partial _{t}v_{n}\left( s,x\right) \right\vert
^{2}dxds+\alpha _{2}\left\vert \frac{\phi ^{\prime }\left( t_{n}\right) }{%
\phi \left( t_{n}\right) }\right\vert ^{2}.%
\end{array}%
\end{equation*}%
.

\item[2$^{nd}$ case] The sequence $\left( t_{n}\right) $ is bounded. Setting%
\begin{equation*}
\lambda _{n}^{2}=\int_{t_{n}}^{t_{n}+T}E_{u_{n}}\left( s\right) ds\text{ and 
}v_{n}\left( t\right) =\frac{1}{\lambda _{n}}u_{n}\left( t+t_{n}\right) .
\end{equation*}%
Using the fact that the sequence $\left( t_{n}\right) $ is bounded and $%
\left( \ref{PHY assumption}\right) ,$\ we infer that there exist $\alpha
_{0},~\alpha _{1}>0,$ such that 
\begin{equation*}
\begin{array}{l}
0<\alpha _{0}\leq \dint_{0}^{T}E_{v_{n}}\left( s\right) ds\leq \alpha _{1},
\\ 
\text{and }\dint_{0}^{T}\dint_{\Omega }\psi \left( x\right) \left\vert
\partial _{t}v_{n}\left( s,x\right) \right\vert ^{2}+\frac{1}{\lambda
_{n}^{2}}\left\vert f_{n}\left( s+t_{n},x\right) \right\vert ^{2}dxds%
\underset{n\rightarrow +\infty }{\longrightarrow }0.%
\end{array}%
\end{equation*}%
To finish the proof we need to proceed as in the first case.
\end{description}
\end{proof}

The result below is a week weighted observability inequality and we need it
to control the L$^{2}$ norm of the solution.

\begin{proposition}
\label{lemma compactness}We assume that $\omega $ and $\mathcal{O}$ satisfy
the GCC and the assumption A1 holds. Let $T>\max \left( T_{\omega },T_{%
\mathcal{O}},T_{0}\right) $ and $\alpha >0.$ Let $\phi $ be a positive
function in $C^{2}\left( 
\mathbb{R}
_{+}\right) $ such that%
\begin{equation}
\begin{array}{l}
\text{for every }I\subset \subset 
\mathbb{R}
_{+},\text{ there exist }m,M>0 \\ 
m\leq \phi \left( s\right) \leq M,\text{ for all }s\in I\text{.}%
\end{array}
\label{rho assumption}
\end{equation}%
In addition, if $\phi ^{\prime }$ is not the null function, we assume that$\ 
$there exists a positive constant $K$ such that%
\begin{equation}
\underset{%
\mathbb{R}
_{+}}{\sup }\left\vert \frac{\phi ^{\prime \prime }\left( t\right) }{\phi
^{\prime }\left( t\right) }\right\vert \leq K.  \label{rho assumption 1}
\end{equation}%
Moreover we suppose that 
\begin{equation}
\begin{array}{c}
\text{ the function\ \ }t\longmapsto \left\vert \frac{\phi ^{\prime }\left(
t\right) }{\phi \left( t\right) }\right\vert \text{ is decreasing and }%
\underset{t\rightarrow +\infty }{\lim }\left\vert \frac{\phi ^{\prime
}\left( t\right) }{\phi \left( t\right) }\right\vert =0.\text{\ \ }%
\end{array}
\label{rho assumption 2}
\end{equation}%
Then there exists $C_{T,\alpha }>0,$ such that for all $\left(
u_{0},v_{0},u_{1},v_{1}\right) \in \mathcal{H},$ and all $t>0,$ the solution
of the system%
\begin{equation}
\left\{ 
\begin{array}{ll}
\partial _{t}^{2}u-\Delta u+b(x)v+a(x)g\left( \partial _{t}u\right) =0 & 
\text{in }\mathbb{R}_{+}^{\ast }\times \Omega \\ 
\partial _{t}^{2}v-\Delta v+b(x)u=0 & \text{in }\mathbb{R}_{+}^{\ast }\times
\Omega \\ 
u=v=0 & \text{on }\mathbb{R}_{+}^{\ast }\times \Gamma \\ 
\left( u\left( 0,x\right) ,\partial _{t}u\left( 0,x\right) \right) =\left(
u_{0},u_{1}\right) \text{ and }\left( v\left( 0,x\right) ,\partial
_{t}v\left( 0,x\right) \right) =\left( v_{0},v_{1}\right) & \text{in }\Omega%
\end{array}%
\right.  \label{compactness lemma system}
\end{equation}%
satisfies the inequality%
\begin{equation}
\begin{array}{l}
\dint_{t}^{t+T}\phi \left( s\right) \dint_{\Omega }\left( \left\vert v\left(
s,x\right) \right\vert ^{2}+\left\vert u\left( s,x\right) \right\vert
^{2}\right) dxds \\ 
\leq C_{T,\alpha }\dint_{t}^{t+T}\phi \left( s\right) \dint_{\Omega }a\left(
x\right) \left( \left\vert g\left( \partial _{t}u\left( s,x\right) \right)
\right\vert ^{2}+\left\vert \partial _{t}u\left( s,x\right) \right\vert
^{2}\right) dxds \\ 
+\alpha \dint_{t}^{t+T}\phi \left( s\right) E_{u,v}\left( s\right) ds.%
\end{array}
\label{compactness lemma estimate}
\end{equation}
\end{proposition}

\begin{proof}
To prove the estimate $\left( \ref{compactness lemma estimate}\right) ,$ we
argue by contradiction. We assume that there exist a positive sequence $%
\left( t_{n}\right) $ and a sequence $\left( U_{n}\left( t\right) =\left(
u_{n}\left( t\right) ,v_{n}\left( t\right) ,\partial _{t}u_{n}\left(
t\right) ,\partial _{t}v_{n}\left( t\right) \right) \right) $ of solutions
of the system $\left( \ref{compactness lemma system}\right) $ with initial
data $\left( u_{0,n},v_{0,n},u_{1,n},v_{1,n}\right) \in \mathcal{H},$ such
that%
\begin{equation}
\begin{array}{l}
\int_{t_{n}}^{t_{n}+T}\phi \left( s\right) \dint_{\Omega }\left\vert
v_{n}\left( s,x\right) \right\vert ^{2}+\left\vert u_{n}\left( s,x\right)
\right\vert ^{2}dxds \\ 
\geq n\dint_{t_{n}}^{t_{n}+T}\phi \left( s\right) \dint_{\Omega }a\left(
x\right) \left( \left\vert g\left( \partial _{t}u_{n}\left( s,x\right)
\right) \right\vert ^{2}+\left\vert \partial _{t}u_{n}\left( s,x\right)
\right\vert ^{2}\right) dxds+\alpha \dint_{t_{n}}^{t_{n}+T}\phi \left(
s\right) E_{u_{n},v_{n}}\left( s\right) ds.%
\end{array}
\label{compactness lemma contradiction estimate}
\end{equation}

\begin{description}
\item[1$^{st}$ case] The sequence $\left( t_{n}\right) $ is bounded. Setting%
\begin{equation*}
\begin{array}{l}
\lambda _{n}^{2}=\dint_{t_{n}}^{t_{n}+T}\dint_{\Omega }\left( \left\vert
v_{n}\left( s,x\right) \right\vert ^{2}+\left\vert u_{n}\left( s,x\right)
\right\vert ^{2}\right) dxds\text{ \ and } \\ 
V_{n}\left( t\right) =\left( 
\begin{array}{cccc}
w_{n}\left( t\right) , & y_{n}\left( t\right) , & \partial _{t}w_{n}\left(
t\right) , & \partial _{t}y_{n}\left( t\right)%
\end{array}%
\right) =\frac{1}{\lambda _{n}}U_{n}\left( t+t_{n}\right) ,%
\end{array}%
\end{equation*}%
Using the fact that the sequence $\left( t_{n}\right) $\ is bounded, $\left( %
\ref{rho assumption}\right) $ and $\left( \ref{compactness lemma
contradiction estimate}\right) ,$\ we infer that there exist $\alpha
_{0},~\alpha _{1},~\alpha _{2}>0,$ such that%
\begin{equation*}
\begin{array}{l}
\dint_{0}^{T}\dint_{\Omega }\left\vert w_{n}\left( s,x\right) \right\vert
^{2}+\left\vert y_{n}\left( s,x\right) \right\vert ^{2}dxds\geq \alpha
_{0}>0, \\ 
\dint_{0}^{T}\dint_{\Omega }a\left( x\right) \left( \left\vert \partial
_{t}w_{n}\left( s,x\right) \right\vert ^{2}+\frac{1}{\lambda _{n}^{2}}%
\left\vert g\left( \partial _{t}u_{n}\left( t_{n}+s,x\right) \right)
\right\vert ^{2}\right) dxds\leq \frac{\alpha _{2}}{n}\text{ } \\ 
\text{and }\dint_{0}^{T}E_{w_{n},y_{n}}\left( s\right) ds\leq \alpha _{1}.%
\end{array}%
\end{equation*}%
To finish the proof we use the unique continuation hypotheses (A1) and we
proceed as in \cite[Proof of lemma 7]{Aloui-daou}.

\item[2$^{nd}$ case] $t_{n}\underset{n\rightarrow +\infty }{\longrightarrow }%
\infty .$ Setting%
\begin{equation*}
\begin{array}{l}
\lambda _{n}^{2}=\int_{t_{n}}^{t_{n}+T}\phi \left( s\right) \dint_{\Omega
}\left( \left\vert v_{n}\left( s,x\right) \right\vert ^{2}+\left\vert
u_{n}\left( s,x\right) \right\vert ^{2}\right) dxds\text{ \ and } \\ 
V_{n}\left( t\right) =\left( 
\begin{array}{cccc}
w_{n}\left( t\right) , & y_{n}\left( t\right) , & \partial _{t}w_{n}\left(
t\right) , & \partial _{t}y_{n}\left( t\right)%
\end{array}%
\right) =\frac{1}{\lambda _{n}}\left( \phi \left( t+t_{n}\right) \right)
^{1/2}U_{n}\left( t+t_{n}\right) ,%
\end{array}%
\end{equation*}%
Therefore%
\begin{equation*}
\left\{ 
\begin{array}{ll}
\partial _{t}^{2}w_{n}-\Delta w_{n}+b(x)y_{n}+\frac{1}{\lambda _{n}}%
a(x)\left( \phi \left( t_{n}+t\right) \right) ^{1/2}g\left( \partial
_{t}u_{n}\right) =f_{n}\left( t,x\right) & \text{in }\mathbb{R}_{+}^{\ast
}\times \Omega \\ 
\partial _{t}^{2}y_{n}-\Delta y_{n}+b(x)w_{n}=f_{n}^{1}\left( t,x\right) & 
\text{in }\mathbb{R}_{+}^{\ast }\times \Omega \\ 
w_{n}=y_{n}=0 & \text{on }\mathbb{R}_{+}^{\ast }\times \Gamma \\ 
\left( w_{n}\left( 0,x\right) ,\partial _{t}w_{n}\left( 0,x\right) \right)
=\left( w_{n,0},w_{n,1}\right) \text{ and }\left( y_{n}\left( 0,x\right)
,\partial _{t}y_{n}\left( 0,x\right) \right) =\left( y_{n,0},y_{n,1}\right)
& \text{in }\Omega%
\end{array}%
\right.
\end{equation*}%
where%
\begin{equation*}
\begin{array}{l}
f_{n}=\frac{1}{2}\phi ^{\prime \prime }\left( t+t_{n}\right) \left( \phi
\left( t+t_{n}\right) \right) ^{-1}w_{n}-\frac{1}{4}\left( \phi ^{\prime
}\left( t+t_{n}\right) \right) ^{2}\phi \left( t+t_{n}\right) ^{-2}w_{n} \\ 
+\frac{1}{\lambda _{n}}\phi ^{\prime }\left( t+t_{n}\right) \left( \phi
\left( t+t_{n}\right) \right) ^{-1/2}\partial _{t}u_{n}, \\ 
f_{n}^{1}=\frac{1}{2}\phi ^{\prime \prime }\left( t+t_{n}\right) \left( \phi
\left( t+t_{n}\right) \right) ^{-1}y_{n}-\frac{1}{4}\left( \phi ^{\prime
}\left( t+t_{n}\right) \right) ^{2}\phi \left( t+t_{n}\right) ^{-2}y_{n} \\ 
+\frac{1}{\lambda _{n}}\phi ^{\prime }\left( t+t_{n}\right) \left( \phi
\left( t+t_{n}\right) \right) ^{-1/2}\partial _{t}v_{n}.%
\end{array}%
\end{equation*}%
Thanks to $\left( \ref{compactness lemma contradiction estimate}\right) $\
we get%
\begin{equation}
\begin{array}{l}
\dint_{0}^{T}\dint_{\Omega }\left\vert w_{n}\left( s,x\right) \right\vert
^{2}+\left\vert y_{n}\left( s,x\right) \right\vert ^{2}dxds=1, \\ 
\frac{1}{\lambda _{n}^{2}}\dint_{0}^{T}\dint_{\Omega }a\left( x\right) \phi
\left( t_{n}+s\right) \left( \left\vert \partial _{t}u_{n}\left(
t_{n}+s,x\right) \right\vert ^{2}+\left\vert g\left( \partial
_{t}u_{n}\left( s+t_{n},x\right) \right) \right\vert ^{2}\right) dxds\leq 
\frac{1}{n} \\ 
\text{and }\dint_{0}^{T}\frac{\phi \left( s+t_{n}\right) }{\lambda _{n}^{2}}%
E_{u_{n},v_{n}}\left( s\right) ds\leq 1/\alpha .%
\end{array}
\label{lemm cotrad 1}
\end{equation}%
Now using the estimates above, we deduce that there exist $\alpha _{1}>0,$
such that 
\begin{equation*}
\begin{array}{c}
\dint_{0}^{T}E_{w_{n},y_{n}}\left( s\right) ds\leq \alpha _{1}.%
\end{array}%
\end{equation*}%
\ Utilizing $\left( \ref{lemm cotrad 1}\right) ,\left( \ref{rho assumption 1}%
\right) $ and $\left( \ref{rho assumption 2}\right) ,$ we can show that%
\begin{equation*}
\begin{array}{c}
\underset{n\rightarrow \infty }{\lim }\dint_{0}^{T}\dint_{\Omega }\left\vert
f_{n}\left( s,x\right) \right\vert ^{2}dxds=\underset{n\rightarrow \infty }{%
\lim }\dint_{0}^{T}\dint_{\Omega }\left\vert f_{n}^{1}\left( s,x\right)
\right\vert ^{2}dxds=0.%
\end{array}%
\end{equation*}%
On the other hand, \ it is easy to see that%
\begin{equation*}
\begin{array}{l}
\dint_{0}^{T}\dint_{\Omega }\psi \left( x\right) \left\vert \partial
_{t}v_{n}\left( s,x\right) \right\vert ^{2}dxds \\ 
\leq \dint_{0}^{T}\dint_{\Omega }\frac{1}{\lambda _{n}^{2}}\phi \left(
t_{n}+s\right) \psi \left( x\right) \left\vert \partial _{t}u_{n}\left(
t_{n}+s,x\right) \right\vert ^{2}+\left\vert \frac{\phi ^{\prime }\left(
t_{n}\right) }{\phi \left( t_{n}\right) }\right\vert ^{2}\psi \left(
x\right) \left\vert w_{n}\left( s,x\right) \right\vert ^{2}dxds \\ 
\leq \dint_{0}^{T}\dint_{\Omega }\frac{1}{\lambda _{n}^{2}}\phi \left(
t_{n}+s\right) \psi \left( x\right) \left\vert \partial _{t}u_{n}\left(
s,x\right) \right\vert ^{2}dxds+\alpha _{2}\left\vert \frac{\phi ^{\prime
}\left( t_{n}\right) }{\phi \left( t_{n}\right) }\right\vert ^{2}\underset{%
n\rightarrow +\infty }{\longrightarrow }0,%
\end{array}%
\end{equation*}%
noting that in the last result we have used the second part of $\left( \ref%
{lemm cotrad 1}\right) $ and $\left( \ref{rho assumption 2}\right) $. To
finish the proof we use the unique continuation hypotheses (A1) and we
proceed as in \cite[Proof of lemma 7]{Aloui-daou}.
\end{description}
\end{proof}

\begin{lemma}
\label{lemma Xt}Let $\varphi \in C^{2}\left( 
\mathbb{R}
_{+}\right) $ and non-decreasing. Let $u$ be a solution of the system $%
\left( \ref{system intro}\right) $ with initial data $U_{0}=\left(
u_{0},v_{0},u_{1},v_{1}\right) \in D\left( \mathcal{A}\right) .$ We set%
\begin{equation}
\begin{array}{l}
X\left( t\right) =\varphi ^{\prime }\left( t\right) \dint_{\Omega }u\left(
t,x\right) \partial _{t}u\left( t,x\right) +v\left( t,x\right) \partial
_{t}v\left( t,x\right) dx \\ 
+k_{1}^{\prime }\varphi \left( t\right) \dint_{\Omega }\partial
_{t}^{2}u\left( t,x\right) \partial _{t}v\left( t,x\right) -\partial
_{t}^{2}v\left( t,x\right) \partial _{t}u\left( t,x\right) dx \\ 
+\varphi \left( t\right) E_{u,v}\left( t\right) +k\varphi ^{\prime }\left(
t\right) E_{\partial _{t}u,\partial _{t}v}\left( t\right) ,%
\end{array}
\label{Xt definition}
\end{equation}%
where $k,~T$ and $k_{1}$ are positive constants. Then%
\begin{equation}
\begin{array}{l}
X\left( t+T\right) -X\left( t\right) +\dint_{t}^{t+T}\varphi ^{\prime
}\left( s\right) E_{u,v}\left( s\right) ds+\dint_{t}^{t+T}\dint_{\Omega
}a\left( x\right) \varphi \left( s\right) g\left( \partial _{t}u\left(
s,x\right) \right) \partial _{t}u\left( s,x\right) dxds \\ 
+k\dint_{t}^{t+T}\dint_{\Omega }a\left( x\right) \varphi ^{\prime }\left(
s\right) g^{\prime }\left( \partial _{t}u\left( s,x\right) \right)
\left\vert \partial _{t}^{2}u\left( s,x\right) \right\vert ^{2}dxds \\ 
\leq 2\dint_{t}^{t+T}\varphi ^{\prime }\left( s\right) \dint_{\Omega
}\left\vert \partial _{t}u\left( s,x\right) \right\vert ^{2}+\left\vert
\partial _{t}v\left( s,x\right) \right\vert ^{2}dxds \\ 
+\dint_{t}^{t+T}\varphi ^{\prime }\left( s\right) \dint_{\Omega }a\left(
x\right) \left( \left\vert g\left( \partial _{t}u\left( s,x\right) \right)
\right\vert ^{2}+\left\vert u\left( s,x\right) \right\vert ^{2}\right) dxds
\\ 
+\left( kE_{\partial _{t}u,\partial _{t}v}\left( 0\right)
+k_{0}E_{u,v}\left( 0\right) \right) \dint_{t}^{t+T}\left\vert \varphi
^{\prime \prime }\left( s\right) \right\vert ds \\ 
+k_{1}\left[ \varphi ^{\prime }\left( t\right) \dint_{\Omega }\partial
_{t}^{2}u\left( s,x\right) \partial _{t}v\left( s,x\right) -\partial
_{t}^{2}v\left( s,x\right) \partial _{t}u\left( s,x\right) dx\right]
_{s=t}^{s=t+T}.%
\end{array}
\label{Xt estimate}
\end{equation}
\end{lemma}

\begin{proof}
We differentiate $X\left( t\right) $ with respect to $t$, we obtain%
\begin{equation}
\begin{array}{l}
X^{\prime }\left( t\right) =\varphi ^{\prime }\left( t\right) \dint_{\Omega
}\left\vert \partial _{t}u\left( t,x\right) \right\vert ^{2}+\left\vert
\partial _{t}v\left( t,x\right) \right\vert ^{2}dx \\ 
+\varphi ^{\prime }\left( t\right) \dint_{\Omega }u\left( t,x\right)
\partial _{t}^{2}u\left( t,x\right) +v\left( t,x\right) \partial
_{t}^{2}v\left( t,x\right) dx \\ 
+\varphi ^{\prime \prime }\left( t\right) \dint_{\Omega }u\left( t,x\right)
\partial _{t}u\left( t,x\right) +v\left( t,x\right) \partial _{t}v\left(
t,x\right) dx \\ 
+k_{1}\frac{d}{dt}\left[ \varphi ^{\prime }\left( t\right) \dint_{\Omega
}\partial _{t}^{2}u\left( t,x\right) \partial _{t}v\left( t,x\right)
-\partial _{t}^{2}v\left( t,x\right) \partial _{t}u\left( t,x\right) dx%
\right] +\varphi ^{\prime }\left( t\right) E_{u,v}\left( t\right) \\ 
-\varphi \left( t\right) \dint_{\Omega }a\left( x\right) g\left( \partial
_{t}u\left( t,x\right) \right) \partial _{t}u\left( t,x\right) dx-k\varphi
^{\prime }\left( t\right) \dint_{\Omega }a\left( x\right) g^{\prime }\left(
\partial _{t}u\left( t,x\right) \right) \left\vert \partial _{t}^{2}u\left(
t,x\right) \right\vert ^{2}dx \\ 
+k\varphi ^{\prime \prime }\left( t\right) E_{\partial _{t}u,\partial
_{t}v}\left( t\right) .%
\end{array}
\label{Xt derivative}
\end{equation}%
Using the first and the second equations of $\left( \ref{system intro}%
\right) ,$ we infer that%
\begin{equation*}
\begin{array}{l}
\varphi ^{\prime }\left( t\right) \dint_{\Omega }\left( u\left( t,x\right)
\partial _{t}^{2}u\left( t,x\right) +v\left( t,x\right) \partial
_{t}^{2}v\left( t,x\right) \right) dx \\ 
=-\varphi ^{\prime }\left( t\right) \dint_{\Omega }\left\vert \nabla u\left(
t,x\right) \right\vert ^{2}+\left\vert \nabla v\left( t,x\right) \right\vert
^{2}+a\left( x\right) g\left( \partial _{t}u\left( t,x\right) \right)
u\left( t,x\right) +2b\left( x\right) u\left( t,x\right) v\left( t,x\right)
dx \\ 
=-2\varphi ^{\prime }\left( t\right) E_{u,v}\left( t\right) +\varphi
^{\prime }\left( t\right) \dint_{\Omega }\left\vert \partial _{t}u\left(
t,x\right) \right\vert ^{2}+\left\vert \partial _{t}v\left( t,x\right)
\right\vert ^{2}dx-\varphi ^{\prime }\left( t\right) \dint_{\Omega }a\left(
x\right) g\left( \partial _{t}u\left( t,x\right) \right) u\left( t,x\right)
dx.%
\end{array}%
\end{equation*}

Thanks to Young's inequality we get%
\begin{equation*}
\begin{array}{c}
\varphi ^{\prime }\left( t\right) \dint_{\Omega }a\left( x\right) g\left(
\partial _{t}u\left( t,x\right) \right) u\left( t,x\right) dx\leq \dfrac{%
\varphi ^{\prime }\left( t\right) }{2}\dint_{\Omega }a\left( x\right) \left(
\left\vert g\left( \partial _{t}u\left( t,x\right) \right) \right\vert
^{2}+\left\vert u\left( t,x\right) \right\vert ^{2}\right) dx.%
\end{array}%
\end{equation*}%
To estimate the third term of the RHS of $\left( \ref{Xt derivative}\right)
, $ we use Poincare's inequality and the fact that the energy is decreasing 
\begin{equation*}
\begin{array}{c}
\varphi ^{\prime \prime }\left( t\right) \dint_{\Omega }u\left( t,x\right)
\partial _{t}u\left( t,x\right) +v\left( t,x\right) \partial _{t}v\left(
t,x\right) dx\leq k_{0}\left\vert \varphi ^{\prime \prime }\left( s\right)
\right\vert E_{u,v}\left( 0\right)%
\end{array}%
\end{equation*}%
For the last term of the RHS of $\left( \ref{Xt derivative}\right) ,$ we use
the fact that%
\begin{equation*}
E_{\partial _{t}u,\partial _{t}v}\left( t\right) \leq E_{\partial
_{t}u,\partial _{t}v}\left( 0\right)
\end{equation*}%
Combining the estimates above, making some arrangement and integrating the
result between $t$ and $t+T$ , we obtain $\left( \ref{Xt estimate}\right) .$
\end{proof}

Let 
\begin{equation}
\begin{array}{c}
H\left( x\right) =\left\{ 
\begin{array}{ll}
h^{-1}\left( x\right) & \text{on }%
\mathbb{R}
_{+} \\ 
\infty & \text{on }%
\mathbb{R}
_{-}^{\ast }.%
\end{array}%
\right.%
\end{array}
\label{H definition}
\end{equation}

Noting that according to \cite{alabau1}, if $h^{-1}$ is a strictly convex $%
C^{1}$ function from $\left[ 0,r_{0}\right] $ to $%
\mathbb{R}
$ such that $h^{-1}\left( 0\right) =\left( h^{-1}\right) ^{\prime }\left(
0\right) =0$. Then the convex conjugate function of $H$ is defined by 
\begin{equation}
\begin{array}{c}
H^{\ast }\left( x\right) =x\left( \left( h^{-1}\right) ^{\prime }\right)
^{-1}\left( x\right) -h^{-1}\left( \left( \left( h^{-1}\right) ^{\prime
}\right) ^{-1}\left( x\right) \right) ,~\text{on }\left[ 0,\left(
h^{-1}\right) ^{\prime }\left( r_{0}\right) \right] .%
\end{array}
\label{convex conjugate}
\end{equation}

\begin{lemma}
\label{proposition phi} We assume that the assumption A2 holds. Let $\varphi 
$ be a solution of the following ODE%
\begin{equation}
\begin{array}{c}
\dfrac{d\varphi }{dt}-\frac{\epsilon _{0}}{2C_{1}}\varphi \left(
h^{-1}\right) ^{\prime }\left( 1/\varphi ^{\beta }\right) =0,~0<\varphi
^{-\beta }\left( 0\right) \leq r_{0}.%
\end{array}
\label{phi equation}
\end{equation}%
where $\epsilon _{0}$ and $C_{1}$ are positive constant. Then we have $%
\varphi $ is a concave strictly increasing function in $C^{3}\left( 
\mathbb{R}
_{+}\right) $. In addition we have

\begin{enumerate}
\item $\underset{t\rightarrow \infty }{\lim }\varphi \left( t\right) =\infty
.$

\item $\underset{t\rightarrow \infty }{\lim }\varphi ^{\prime }\left(
t\right) =\frac{\epsilon _{0}\alpha _{0}}{2C_{1}}$ and $\underset{%
t\rightarrow \infty }{\lim }\left\vert \dfrac{\varphi ^{\prime \prime
}\left( t\right) }{\varphi ^{\prime }\left( t\right) }\right\vert =0$.

\item The function \ $t\longmapsto \left\vert \frac{\varphi ^{\prime \prime
}\left( t\right) }{\varphi ^{\prime }\left( t\right) }\right\vert $ is
decreasing.

\item If $\varphi ^{\prime \prime }$ is not the null function, then there
exists $K>0$ such that $\underset{%
\mathbb{R}
_{+}}{\sup }\left\vert \frac{\varphi ^{\prime \prime \prime }\left( t\right) 
}{\varphi ^{\prime \prime }\left( t\right) }\right\vert \leq K.$

\item $\dint_{0}^{\infty }\left\vert \varphi ^{\prime \prime }\left(
s\right) \right\vert ds=\varphi ^{\prime }\left( 0\right) -\frac{\epsilon
_{0}\alpha _{0}}{2C_{1}}.$

\item $\frac{\epsilon _{0}}{2C_{1}}\dint_{0}^{\infty }\varphi \left(
s\right) H^{\ast }\left( \frac{2C_{1}\varphi ^{\prime }\left( s\right) }{%
\epsilon _{0}\varphi \left( s\right) }\right) ds\leq \frac{1}{\beta -1}%
\varphi ^{1-\beta }\left( 0\right) .$
\end{enumerate}
\end{lemma}

\begin{proof}
Using the second part of $\left( \ref{assumption2}\right) $ and the fact that%
\begin{equation}
0<\varphi ^{-\beta }\left( t\right) \leq \varphi ^{-\beta }\left( 0\right)
\leq r_{0},\text{ for all }t>0,  \label{phi t beta bound}
\end{equation}%
we obtain%
\begin{equation}
\begin{array}{c}
\varphi ^{\prime \prime }\left( t\right) =\frac{\epsilon _{0}}{2C_{1}}%
\varphi ^{\prime }\left( t\right) \left( \left( h^{-1}\right) ^{\prime
}\left( \varphi ^{-\beta }\left( t\right) \right) -\beta \varphi ^{-\beta
}\left( t\right) \left( h^{-1}\right) ^{\prime \prime }\left( \varphi
^{-\beta }\left( t\right) \right) \right) \leq 0,%
\end{array}
\label{concave phi}
\end{equation}%
for all $t\in 
\mathbb{R}
_{+},$ which means that the function $\varphi $ is a concave on $%
\mathbb{R}
_{+}.$

It is easy to see that the function $\varphi \in C^{3}\left( 
\mathbb{R}
_{+}\right) $.

\begin{enumerate}
\item First we note that 
\begin{equation*}
\begin{array}{c}
\varphi ^{-1}\left( t\right) =\frac{2C_{1}}{\epsilon _{0}}\dint_{\varphi
\left( 0\right) }^{t}\dfrac{ds}{s\left( h^{-1}\right) ^{\prime }\left(
1/s^{\beta }\right) }.%
\end{array}%
\end{equation*}%
Therefore, using the fact that%
\begin{equation*}
\begin{array}{c}
s\left( h^{-1}\right) ^{\prime }\left( 1/s^{\beta }\right) \leq s\left(
h^{-1}\right) ^{\prime }\left( 1/\varphi ^{\beta }\left( 0\right) \right) ,%
\text{ for all }s\in \lbrack 0,\infty ),%
\end{array}%
\end{equation*}%
we deduce that%
\begin{equation*}
\begin{array}{c}
\underset{t\rightarrow \infty }{\lim }\varphi ^{-1}\left( t\right) =\infty ,%
\end{array}%
\end{equation*}%
thus 
\begin{equation}
\begin{array}{c}
\underset{t\rightarrow \infty }{\lim }\varphi \left( t\right) =\infty .%
\end{array}
\label{limit phi}
\end{equation}

\item Thanks to $\left( \ref{limit phi}\right) $ and $\left( \ref%
{assumption2}\right) $ we get 
\begin{equation}
\begin{array}{c}
\underset{t\rightarrow \infty }{\lim }\varphi ^{\prime }\left( t\right) =%
\frac{\epsilon _{0}\alpha _{0}}{2C_{1}}.%
\end{array}
\label{limit phi derivative}
\end{equation}

Direct computations and using $\left( \ref{assumption1}\right) ,$ yield%
\begin{equation*}
\begin{array}{c}
\underset{t\rightarrow \infty }{\lim }\left\vert \dfrac{\varphi ^{\prime
\prime }\left( t\right) }{\varphi ^{\prime }\left( t\right) }\right\vert =0.%
\end{array}%
\end{equation*}

\item From $\left( \ref{concave phi}\right) ,$ we deduce that%
\begin{equation*}
\begin{array}{c}
\left\vert \frac{\varphi ^{\prime \prime }\left( t\right) }{\varphi ^{\prime
}\left( t\right) }\right\vert =\frac{\epsilon _{0}}{2C_{1}}\left( -\left(
h^{-1}\right) ^{\prime }\left( \varphi ^{-\beta }\left( t\right) \right)
+\beta \varphi ^{-\beta }\left( t\right) \left( h^{-1}\right) ^{\prime
\prime }\left( \varphi ^{-\beta }\left( t\right) \right) \right) .%
\end{array}%
\end{equation*}%
We differentiate the estimate above and making some arrangements, we obtain%
\begin{equation*}
\begin{array}{c}
\frac{d}{dt}\left\vert \frac{\varphi ^{\prime \prime }\left( t\right) }{%
\varphi ^{\prime }\left( t\right) }\right\vert =-\frac{\epsilon _{0}\varphi
^{\prime }\left( t\right) }{2C_{1}\varphi ^{\beta +1}\left( t\right) }\left(
\left( \beta ^{2}-\beta \right) \left( h^{-1}\right) ^{\prime \prime }\left(
\varphi ^{-\beta }\left( t\right) \right) +\beta ^{2}\varphi ^{-\beta
}\left( t\right) \left( h^{-1}\right) ^{\prime \prime \prime }\left( \varphi
^{-\beta }\left( t\right) \right) \right) .%
\end{array}%
\end{equation*}%
\ From the estimate above, $\left( \ref{phi t beta bound}\right) $\ and $%
\left( \ref{assumption2}\right) ,$ we see that the function \ $t\longmapsto
\left\vert \frac{\varphi ^{\prime \prime }\left( t\right) }{\varphi ^{\prime
}\left( t\right) }\right\vert $ is decreasing.

\item We differentiate the identity $\left( \ref{concave phi}\right) $ and
making some arrangements, we obtain%
\begin{equation*}
\begin{array}{l}
\frac{\varphi ^{\prime \prime \prime }\left( t\right) }{\varphi ^{\prime
\prime }\left( t\right) }=\frac{\epsilon _{0}}{2C_{1}}\left( \left(
h^{-1}\right) ^{\prime }\left( \varphi ^{-\beta }\left( t\right) \right)
-\beta \varphi ^{-\beta }\left( t\right) \left( h^{-1}\right) ^{\prime
\prime }\left( \varphi ^{-\beta }\left( t\right) \right) \right) \\ 
+\frac{\epsilon _{0}\left( \varphi ^{\prime }\left( t\right) \right) ^{2}}{%
2C_{1}\varphi ^{\prime \prime }\left( t\right) \varphi \left( t\right) }%
\left( \left( \beta ^{2}-\beta \right) \varphi ^{-\beta }\left( t\right)
\left( h^{-1}\right) ^{\prime \prime }\left( \varphi ^{-\beta }\left(
t\right) \right) +\beta ^{2}\varphi ^{-2\beta }\left( t\right) \left(
h^{-1}\right) ^{\prime \prime \prime }\left( \varphi ^{-\beta }\left(
t\right) \right) \right) .%
\end{array}%
\end{equation*}%
On the other hand, from $\left( \ref{phi equation}\right) $ and $\left( \ref%
{concave phi}\right) ,$ we infer that%
\begin{equation*}
\begin{array}{c}
\left\vert \frac{\left( \varphi ^{\prime }\left( t\right) \right) ^{2}}{%
\varphi ^{\prime \prime }\left( t\right) \varphi \left( t\right) }%
\right\vert =\frac{\left( h^{-1}\right) ^{\prime }\left( \varphi ^{-\beta
}\left( t\right) \right) }{\beta \varphi ^{-\beta }\left( t\right) \left(
h^{-1}\right) ^{\prime \prime }\left( \varphi ^{-\beta }\left( t\right)
\right) -\left( h^{-1}\right) ^{\prime }\left( \varphi ^{-\beta }\left(
t\right) \right) }.%
\end{array}%
\end{equation*}%
Combining the two estimates above, we see that%
\begin{equation*}
\begin{array}{l}
\left\vert \frac{\varphi ^{\prime \prime \prime }\left( t\right) }{\varphi
^{\prime \prime }\left( t\right) }\right\vert \leq \frac{\epsilon _{0}}{%
2C_{1}}\left( \beta \varphi ^{-\beta }\left( t\right) \left( h^{-1}\right)
^{\prime \prime }\left( \varphi ^{-\beta }\left( t\right) \right) -\left(
h^{-1}\right) ^{\prime }\left( \varphi ^{-\beta }\left( t\right) \right)
\right) \\ 
+\frac{\epsilon _{0}\left( h^{-1}\right) ^{\prime }\left( \varphi ^{-\beta
}\left( t\right) \right) }{2C_{1}}\left( \frac{\left( \beta ^{2}-\beta
\right) \varphi ^{-\beta }\left( t\right) \left( h^{-1}\right) ^{\prime
\prime }\left( \varphi ^{-\beta }\left( t\right) \right) +\beta ^{2}\varphi
^{-2\beta }\left( t\right) \left( h^{-1}\right) ^{\prime \prime \prime
}\left( \varphi ^{-\beta }\left( t\right) \right) }{\beta \varphi ^{-\beta
}\left( t\right) \left( h^{-1}\right) ^{\prime \prime }\left( \varphi
^{-\beta }\left( t\right) \right) -\left( h^{-1}\right) ^{\prime }\left(
\varphi ^{-\beta }\left( t\right) \right) }\right) .%
\end{array}%
\end{equation*}%
So from Assumption A2, we conclude that there exists $K>0$ such that%
\begin{equation*}
\underset{%
\mathbb{R}
_{+}}{\sup }\left\vert \frac{\varphi ^{\prime \prime \prime }\left( t\right) 
}{\varphi ^{\prime \prime }\left( t\right) }\right\vert \leq K.
\end{equation*}

\item Using $\left( \ref{concave phi}\right) $ and $\left( \ref{limit phi
derivative}\right) ,$ we obtain%
\begin{equation*}
\begin{array}{c}
\dint_{0}^{\infty }\left\vert \varphi ^{\prime \prime }\left( s\right)
\right\vert ds=-\dint_{0}^{\infty }\varphi ^{\prime \prime }\left( s\right)
ds=\varphi ^{\prime }\left( 0\right) -\frac{\epsilon _{0}\alpha _{0}}{2C_{1}}%
.%
\end{array}%
\end{equation*}

\item Thanks to $\left( \ref{convex conjugate}\right) $ and $\left( \ref{phi
equation}\right) ,$ we see that%
\begin{equation*}
\begin{array}{c}
\varphi \left( s\right) H^{\ast }\left( \frac{2C_{1}\varphi ^{\prime }\left(
s\right) }{\epsilon _{0}\varphi \left( s\right) }\right) \leq \frac{2C_{1}}{%
\epsilon _{0}}\varphi ^{\prime }\left( s\right) \left( \left( h^{-1}\right)
^{\prime }\right) ^{-1}\left( \frac{2C_{1}\varphi ^{\prime }\left( s\right) 
}{\epsilon _{0}\varphi \left( s\right) }\right) =\frac{2C_{1}}{\epsilon _{0}}%
\dfrac{\varphi ^{\prime }\left( s\right) }{\varphi ^{\beta }\left( s\right) }%
,%
\end{array}%
\end{equation*}%
therefore integrating the estimate above between zero and infinity and using 
$\left( \ref{limit phi}\right) $\ and the fact that $\beta >1$ we obtain 
\begin{equation*}
\begin{array}{c}
\frac{\epsilon _{0}}{2C_{1}}\dint_{0}^{\infty }\varphi \left( s\right)
H^{\ast }\left( \frac{2C_{1}\varphi ^{\prime }\left( s\right) }{\epsilon
_{0}\varphi \left( s\right) }\right) ds\leq \frac{1}{\beta -1}\varphi
^{1-\beta }\left( 0\right) .%
\end{array}%
\end{equation*}
\end{enumerate}
\end{proof}

\subsection{Proof of Theorem 1}

We assume that $\omega $ and $\mathcal{O}$ satisfy the GCC and the
assumption A1 holds$.$ Let $\left( u,v\right) $ be a solution of the system $%
\left( \ref{system intro}\right) $ with initial data $U_{0}=\left(
u_{0},v_{0},u_{1},v_{1}\right) \in D\left( \mathcal{A}\right) .$ Let $T>\max
\left( T_{\omega },T_{\mathcal{O}},T_{0}\right) .$

We have $u$ is a solution of the nonhomogeneous wave equation with a
localized nonlinear damping and $\left( \omega ,T\right) $ satisfies the
GCC. In addition, taking into account of lemma $\left( \ref{proposition phi}%
\right) ,$ we see that $\varphi ^{\prime }$ satisfies the required
assumptions of proposition $\left( \ref{observability lemma}\right) .$
Therefore, using the observability estimate $\left( \ref{observability by a}%
\right) $ and $\left( \ref{Xt estimate}\right) ,$ we deduce that 
\begin{equation}
\begin{array}{l}
X\left( t+T\right) -X\left( t\right) +\dint_{t}^{t+T}\varphi ^{\prime
}\left( s\right) E_{u,v}\left( s\right) ds+\dint_{t}^{t+T}\dint_{\Omega
}a\left( x\right) \varphi \left( s\right) g\left( \partial _{t}u\left(
s,x\right) \right) \partial _{t}u\left( s,x\right) dxds \\ 
+k\dint_{t}^{t+T}\dint_{\Omega }a\left( x\right) \varphi ^{\prime }\left(
s\right) g^{\prime }\left( \partial _{t}u\left( s,x\right) \right)
\left\vert \partial _{t}^{2}u\left( s,x\right) \right\vert ^{2}dxds \\ 
\leq 2\dint_{t}^{t+T}\varphi ^{\prime }\left( s\right) \dint_{\Omega
}\left\vert \partial _{t}v\left( s,x\right) \right\vert ^{2}dxds \\ 
+\left( 4C_{T}+1\right) \left( \dint_{t}^{t+T}\varphi ^{\prime }\left(
s\right) \dint_{\Omega }\left( a\left( x\right) \left( \left\vert g\left(
\partial _{t}u\left( s,x\right) \right) \right\vert ^{2}+\left\vert \partial
_{t}u\left( s,x\right) \right\vert ^{2}\right) +\left\vert b\left( x\right)
v\left( s,x\right) \right\vert ^{2}\right) dxds\right) \\ 
+3\dint_{t}^{t+T}\varphi ^{\prime }\left( s\right) \dint_{\Omega }a\left(
x\right) \left\vert u\left( s,x\right) \right\vert ^{2}dxds+\left(
kE_{\partial _{t}u,\partial _{t}v}\left( 0\right) +k_{0}E_{u,v}\left(
0\right) \right) \dint_{t}^{t+T}\left\vert \varphi ^{\prime \prime }\left(
s\right) \right\vert ds \\ 
+k_{1}\left[ \varphi ^{\prime }\left( t\right) \dint_{\Omega }\partial
_{t}^{2}u\left( s,x\right) \partial _{t}v\left( s,x\right) -\partial
_{t}^{2}v\left( s,x\right) \partial _{t}u\left( s,x\right) dx\right]
_{s=t}^{s=t+T}.%
\end{array}
\label{proof theorem Xt estimate}
\end{equation}

To estimate $\dint_{t}^{t+T}\varphi ^{\prime }\left( s\right) \dint_{\Omega
}\left\vert \partial _{t}v\left( s,x\right) \right\vert ^{2}dxds,$ we first
use the fact that $v$ is a solution of the nonhomogeneous wave equation and $%
\left( \mathcal{O},T\right) $ satisfies the GCC, then from the observability
estimate $\left( \ref{observability by a}\right) ,$ we infer that 
\begin{equation}
\begin{array}{c}
2\dint_{t}^{t+T}\varphi ^{\prime }\left( s\right) \dint_{\Omega }\left\vert
\partial _{t}v\left( s,x\right) \right\vert ^{2}dxds\leq 4C_{T}\left(
\dint_{t}^{t+T}\varphi ^{\prime }\left( s\right) \dint_{\Omega }b\left(
x\right) \left\vert \partial _{t}v\left( s,x\right) \right\vert
^{2}+\left\vert b\left( x\right) u\left( s,x\right) \right\vert
^{2}dxds\right) .%
\end{array}
\label{proof theorem estimate dtv}
\end{equation}

Now we estimate $\dint_{t}^{t+T}\varphi ^{\prime }\left( s\right)
\dint_{\Omega }b\left( x\right) \left\vert \partial _{t}v\left( s,x\right)
\right\vert ^{2}dxds.$ We have 
\begin{equation}
\left\{ 
\begin{array}{ll}
\partial _{t}^{2}\left( \partial _{t}u\right) -\Delta \left( \partial
_{t}u\right) +b(x)\left( \partial _{t}v\right) +a(x)g^{\prime }\left(
\partial _{t}u\right) \partial _{t}^{2}u=0 & \text{in }\mathbb{R}_{+}^{\ast
}\times \Omega \\ 
\partial _{t}^{2}\left( \partial _{t}v\right) -\Delta \left( \partial
_{t}v\right) +b(x)\left( \partial _{t}u\right) =0 & \text{in }\mathbb{R}%
_{+}^{\ast }\times \Omega \\ 
\partial _{t}u=\partial _{t}v=0 & \text{on }\mathbb{R}_{+}^{\ast }\times
\Gamma .%
\end{array}%
\right.  \label{system derivative}
\end{equation}%
We multiply the first equation of $\left( \ref{system derivative}\right) $
by $\left( \varphi ^{\prime }\left( t\right) \partial _{t}v\right) $ and the
second equation by $\left( \varphi ^{\prime }\left( t\right) \partial
_{t}u\right) $ and integrating the difference of these results over $\Omega $%
, we obtain%
\begin{equation*}
\begin{array}{l}
\varphi ^{\prime }\left( t\right) \dint_{\Omega }b\left( x\right) \left\vert
\partial _{t}v\left( t,x\right) \right\vert ^{2}dx=-\frac{d}{dt}\left(
\varphi ^{\prime }\left( t\right) \dint_{\Omega }\left( \partial _{t}v\left(
t,x\right) \partial _{t}^{2}u\left( t,x\right) -\partial _{t}u\left(
t,x\right) \partial _{t}^{2}v\left( t,x\right) \right) dx\right) \\ 
-\varphi ^{\prime \prime }\left( t\right) \dint_{\Omega }\left( \partial
_{t}v\left( t,x\right) \partial _{t}^{2}u\left( t,x\right) -\partial
_{t}u\left( t,x\right) \partial _{t}^{2}v\left( t,x\right) \right) dx \\ 
+\varphi ^{\prime }\left( t\right) \dint_{\Omega }b\left( x\right)
\left\vert \partial _{t}u\left( t,x\right) \right\vert ^{2}dx-\varphi
^{\prime }\left( t\right) \dint_{\Omega }a\left( x\right) g^{\prime }\left(
\partial _{t}u\left( t,x\right) \right) \partial _{t}^{2}u\left( t,x\right)
\partial _{t}v\left( t,x\right) dx.%
\end{array}%
\end{equation*}%
Using Young's inequality, we infer that%
\begin{equation*}
\begin{array}{l}
\varphi ^{\prime }\left( t\right) \dint_{\Omega }b\left( x\right) \left\vert
\partial _{t}v\left( s,x\right) \right\vert ^{2}dx\leq \epsilon \varphi
^{\prime }\left( t\right) \dint_{\Omega }a\left( x\right) \left\vert
\partial _{t}v\left( s,x\right) \right\vert ^{2}dx \\ 
-\dfrac{d}{dt}\left( \varphi ^{\prime }\left( t\right) \dint_{\Omega }\left(
\partial _{t}u\left( t,x\right) \partial _{t}^{2}v\left( t,x\right)
-\partial _{t}v\left( t,x\right) \partial _{t}^{2}u\left( t,x\right) \right)
dx\right) \\ 
+\varphi ^{\prime \prime }\left( t\right) \dint_{\Omega }\left( \partial
_{t}u\left( t,x\right) \partial _{t}^{2}v\left( t,x\right) -\partial
_{t}v\left( t,x\right) \partial _{t}^{2}u\left( t,x\right) \right) dx \\ 
+\varphi ^{\prime }\left( t\right) \dint_{\Omega }b\left( x\right)
\left\vert \partial _{t}u\left( s,x\right) \right\vert ^{2}dx+\frac{1}{%
\epsilon }\varphi ^{\prime }\left( t\right) \dint_{\Omega }a\left( x\right)
\left\vert g^{\prime }\left( \partial _{t}u\left( t,x\right) \right)
\partial _{t}^{2}u\left( s,x\right) \right\vert ^{2}dx.%
\end{array}%
\end{equation*}%
Integrating the estimate above between $t$ and $t+T,$ we obtain%
\begin{equation*}
\begin{array}{l}
\dint_{t}^{t+T}\varphi ^{\prime }\left( s\right) \dint_{\Omega }b\left(
x\right) \left\vert \partial _{t}v\left( s,x\right) \right\vert ^{2}dxds \\ 
\leq \epsilon \dint_{t}^{t+T}\varphi ^{\prime }\left( s\right) \dint_{\Omega
}a\left( x\right) \left\vert \partial _{t}v\left( s,x\right) \right\vert
^{2}dxds+\dint_{t}^{t+T}\varphi ^{\prime }\left( s\right) \dint_{\Omega
}b\left( x\right) \left\vert \partial _{t}u\left( s,x\right) \right\vert
^{2}dxds \\ 
+\frac{1}{\epsilon }\dint_{t}^{t+T}\varphi ^{\prime }\left( s\right)
\dint_{\Omega }a\left( x\right) \left\vert g^{\prime }\left( \partial
_{t}u\left( s,x\right) \right) \partial _{t}^{2}u\left( s,x\right)
\right\vert ^{2}dxds \\ 
+\dint_{t}^{t+T}\varphi ^{\prime \prime }\left( s\right) \dint_{\Omega
}\left( \partial _{t}u\left( t,x\right) \partial _{t}^{2}v\left( t,x\right)
-\partial _{t}v\left( t,x\right) \partial _{t}^{2}u\left( t,x\right) \right)
dxds \\ 
-\left[ \varphi ^{\prime }\left( s\right) \dint_{\Omega }\left( \partial
_{t}^{2}u\left( s,x\right) \partial _{t}v\left( s,x\right) -\partial
_{t}^{2}v\left( s,x\right) \partial _{t}u\left( s,x\right) \right) dx\right]
_{s=t}^{s=t+T}.%
\end{array}%
\end{equation*}%
Now using the observability estimate $\left( \ref{proof theorem estimate dtv}%
\right) $ and taking $\epsilon =\frac{1}{4\left\Vert a\right\Vert _{\infty
}C_{T}},$ we get%
\begin{equation*}
\begin{array}{l}
\dint_{t}^{t+T}\varphi ^{\prime }\left( s\right) \dint_{\Omega }b\left(
x\right) \left\vert \partial _{t}v\left( s,x\right) \right\vert ^{2}dxds \\ 
\leq 2\dint_{t}^{t+T}\varphi ^{\prime }\left( s\right) \dint_{\Omega
}b\left( x\right) \left\vert \partial _{t}u\left( s,x\right) \right\vert
^{2}+\left\vert b\left( x\right) u\left( s,x\right) \right\vert ^{2}dxds \\ 
+16\left\Vert a\right\Vert _{\infty }C_{T}\dint_{t}^{t+T}\varphi ^{\prime
}\left( s\right) \dint_{\Omega }a\left( x\right) \left\vert g^{\prime
}\left( \partial _{t}u\left( s,x\right) \right) \partial _{t}^{2}u\left(
s,x\right) \right\vert ^{2}dxds \\ 
+2\dint_{t}^{t+T}\varphi ^{\prime \prime }\left( s\right) \dint_{\Omega
}\left( \partial _{t}u\left( t,x\right) \partial _{t}^{2}v\left( t,x\right)
-\partial _{t}v\left( t,x\right) \partial _{t}^{2}u\left( t,x\right) \right)
dxds \\ 
-2\left[ \varphi ^{\prime }\left( s\right) \dint_{\Omega }\left( \partial
_{t}^{2}u\left( s,x\right) \partial _{t}v\left( s,x\right) -\partial
_{t}^{2}v\left( s,x\right) \partial _{t}u\left( s,x\right) \right) dx\right]
_{s=t}^{s=t+T}.%
\end{array}%
\end{equation*}%
Combining the estimate above and $\left( \ref{proof theorem estimate dtv}%
\right) ,$ we find that%
\begin{equation}
\begin{array}{l}
2\dint_{t}^{t+T}\varphi ^{\prime }\left( s\right) \dint_{\Omega }\left\vert
\partial _{t}v\left( s,x\right) \right\vert ^{2}dxds \\ 
\leq 8C_{T}\dint_{t}^{t+T}\varphi ^{\prime }\left( s\right) \dint_{\Omega
}b\left( x\right) \left\vert \partial _{t}u\left( s,x\right) \right\vert
^{2}+\left\vert b\left( x\right) u\left( s,x\right) \right\vert ^{2}dxds \\ 
+64\left\Vert a\right\Vert _{\infty }\left( C_{T}\right) ^{2}\left\Vert
g^{\prime }\right\Vert _{L^{\infty }}\int_{t}^{t+T}\varphi ^{\prime }\left(
s\right) \dint_{\Omega }a\left( x\right) g^{\prime }\left( \partial
_{t}u\left( s,x\right) \right) \left\vert \partial _{t}^{2}u\left(
s,x\right) \right\vert ^{2}dxds \\ 
+8C_{T}\left( \sum_{i=0}^{1}E_{\partial _{t}^{i}u,\partial _{t}^{i}v}\left(
0\right) \right) \dint_{t}^{t+T}\left\vert \varphi "\left( s\right)
\right\vert ds \\ 
-8C_{T}\left[ \varphi ^{\prime }\left( s\right) \dint_{\Omega }\left(
\partial _{t}^{2}u\left( s,x\right) \partial _{t}v\left( s,x\right)
-\partial _{t}^{2}v\left( s,x\right) \partial _{t}u\left( s,x\right) \right)
dx\right] _{s=t}^{s=t+T}.%
\end{array}
\label{proof theorem estimate dtv 1}
\end{equation}%
Now using the observability estimate $\left( \ref{observability by a}\right)
,$ we infer that%
\begin{equation}
\begin{array}{l}
2\dint_{t}^{t+T}\varphi ^{\prime }\left( s\right) \dint_{\Omega }\left\vert
\partial _{t}v\left( s,x\right) \right\vert ^{2}dxds \\ 
\leq 32C_{T}^{2}\left\Vert b\right\Vert _{\infty }\left(
\dint_{t}^{t+T}\varphi ^{\prime }\left( s\right) \dint_{\Omega }\left(
a\left( x\right) \left( \left\vert g\left( \partial _{t}u\left( s,x\right)
\right) \right\vert ^{2}+\left\vert \partial _{t}u\left( s,x\right)
\right\vert ^{2}\right) \right) dxds\right) \\ 
+\left( 8C_{T}\left\Vert b\right\Vert _{\infty }+32C_{T}^{2}\left\Vert
b\right\Vert _{\infty }^{2}\right) \dint_{t}^{t+T}\varphi ^{\prime }\left(
s\right) \dint_{\Omega }\left\vert u\left( s,x\right) \right\vert
^{2}+\left\vert v\left( s,x\right) \right\vert ^{2}dxds \\ 
+64\left\Vert a\right\Vert _{\infty }\left( C_{T}\right) ^{2}\left\Vert
g^{\prime }\right\Vert _{L^{\infty }}\int_{t}^{t+T}\varphi ^{\prime }\left(
s\right) \dint_{\Omega }a\left( x\right) g^{\prime }\left( \partial
_{t}u\left( s,x\right) \right) \left\vert \partial _{t}^{2}u\left(
s,x\right) \right\vert ^{2}dxds \\ 
+8C_{T}\left( \sum_{i=0}^{1}E_{\partial _{t}^{i}u,\partial _{t}^{i}v}\left(
0\right) \right) \dint_{t}^{t+T}\left\vert \varphi "\left( s\right)
\right\vert ds \\ 
-8C_{T}\left[ \varphi ^{\prime }\left( s\right) \dint_{\Omega }\left(
\partial _{t}^{2}u\left( s,x\right) \partial _{t}v\left( s,x\right)
-\partial _{t}^{2}v\left( s,x\right) \partial _{t}u\left( s,x\right) \right)
dx\right] _{s=t}^{s=t+T}.%
\end{array}
\label{proof theorem estimate dtv final}
\end{equation}%
Utilizing $\left( \ref{proof theorem Xt estimate}\right) ,$ $\left( \ref%
{proof theorem estimate dtv final}\right) $ and making some arrangements, we
obtain%
\begin{equation*}
\begin{array}{l}
X\left( t+T\right) -X\left( t\right) +\dint_{t}^{t+T}\varphi ^{\prime
}\left( s\right) E_{u,v}\left( s\right) ds+\dint_{t}^{t+T}\varphi \left(
s\right) \dint_{\Omega }a\left( x\right) g\left( \partial _{t}u\left(
s,x\right) \right) \partial _{t}u\left( s,x\right) dxds \\ 
+k\dint_{t}^{t+T}\varphi ^{\prime }\left( s\right) \dint_{\Omega }a\left(
x\right) g^{\prime }\left( \partial _{t}u\left( s,x\right) \right)
\left\vert \partial _{t}^{2}u\left( s,x\right) \right\vert ^{2}dxds \\ 
\leq \left( 4C_{T}+32C_{T}^{2}\left\Vert b\right\Vert _{\infty }+1\right)
\left( \dint_{t}^{t+T}\varphi ^{\prime }\left( s\right) \dint_{\Omega
}\left( a\left( x\right) \left( \left\vert g\left( \partial _{t}u\left(
s,x\right) \right) \right\vert ^{2}+\left\vert \partial _{t}u\left(
s,x\right) \right\vert ^{2}\right) \right) dxds\right) \\ 
+\left( \left\Vert a\right\Vert _{\infty }+12C_{T}\left\Vert b\right\Vert
_{\infty }+32C_{T}^{2}\left\Vert b\right\Vert _{\infty }^{2}\right)
\dint_{t}^{t+T}\varphi ^{\prime }\left( s\right) \dint_{\Omega }\left\vert
v\left( s,x\right) \right\vert ^{2}+\left\vert u\left( s,x\right)
\right\vert ^{2}dxds \\ 
+64\left\Vert a\right\Vert _{\infty }\left( C_{T}\right) ^{2}\left\Vert
g^{\prime }\right\Vert _{L^{\infty }}\dint_{t}^{t+T}\varphi ^{\prime }\left(
s\right) \dint_{\Omega }a\left( x\right) g^{\prime }\left( \partial
_{t}u\left( s,x\right) \right) \left\vert \partial _{t}^{2}u\left(
s,x\right) \right\vert ^{2}dxds \\ 
+\left( k_{1}-8C_{T}\right) \left[ \varphi ^{\prime }\left( s\right)
\dint_{\Omega }\left( \partial _{t}^{2}u\left( s,x\right) \partial
_{t}v\left( s,x\right) -\partial _{t}^{2}v\left( s,x\right) \partial
_{t}u\left( s,x\right) \right) dx\right] _{s=t}^{s=t+T} \\ 
+\left( 8C_{T}+k\right) \left( \sum_{i=0}^{1}E_{\partial _{t}^{i}u,\partial
_{t}^{i}v}\left( 0\right) \right) \dint_{t}^{t+T}\left\vert \varphi ^{\prime
\prime }\left( s\right) \right\vert ds.%
\end{array}%
\end{equation*}

We take $k_{1}=8C_{T},$ we conclude that%
\begin{equation}
\begin{array}{l}
X\left( t+T\right) -X\left( t\right) +\dint_{t}^{t+T}\varphi ^{\prime
}\left( s\right) E_{u,v}\left( s\right) ds+\dint_{t}^{t+T}\varphi \left(
s\right) \dint_{\Omega }a\left( x\right) g\left( \partial _{t}u\left(
s,x\right) \right) \partial _{t}u\left( s,x\right) dxds \\ 
+\left( k-64\left\Vert a\right\Vert _{\infty }\left( C_{T}\right)
^{2}\left\Vert g^{\prime }\right\Vert _{L^{\infty }}\right)
\dint_{t}^{t+T}\varphi ^{\prime }\left( s\right) \dint_{\Omega }a\left(
x\right) g^{\prime }\left( \partial _{t}u\left( s,x\right) \right)
\left\vert \partial _{t}^{2}u\left( s,x\right) \right\vert ^{2}dxds \\ 
\leq \left( \left\Vert a\right\Vert _{\infty }+12C_{T}\left\Vert
b\right\Vert _{\infty }+32C_{T}^{2}\left\Vert b\right\Vert _{\infty
}^{2}\right) \dint_{t}^{t+T}\varphi ^{\prime }\left( s\right) \dint_{\Omega
}\left\vert v\left( s,x\right) \right\vert ^{2}+\left\vert u\left(
s,x\right) \right\vert ^{2}dxds \\ 
+\left( 32\left( C_{T}\right) ^{2}\left\Vert b\right\Vert _{\infty
}+4C_{T}+1\right) \left( \dint_{t}^{t+T}\varphi ^{\prime }\left( s\right)
\dint_{\Omega }a\left( x\right) \left( \left\vert \partial _{t}u\left(
s,x\right) \right\vert ^{2}+\left\vert g\left( \partial _{t}u\left(
s,x\right) \right) \right\vert ^{2}\right) dxds\right) \\ 
+\left( 8C_{T}+k+k_{0}\right) \left( \sum_{i=0}^{1}E_{\partial
_{t}^{i}u,\partial _{t}^{i}v}\left( 0\right) \right)
\dint_{t}^{t+T}\left\vert \varphi ^{\prime \prime }\left( s\right)
\right\vert ds.%
\end{array}
\label{proof theorem Xt prefinal estimate}
\end{equation}

Now using $\left( \ref{compactness lemma estimate}\right) $ with $\alpha =%
\frac{1}{2\left( \left\Vert a\right\Vert _{\infty }+12C_{T}\left\Vert
b\right\Vert _{\infty }+32C_{T}^{2}\left\Vert b\right\Vert _{\infty
}^{2}\right) }$ , we get%
\begin{equation}
\begin{array}{l}
X\left( t+T\right) +\frac{1}{2}\dint_{t}^{t+T}\varphi ^{\prime }\left(
s\right) E_{u,v}\left( s\right) ds+\dint_{t}^{t+T}\varphi \left( s\right)
\dint_{\Omega }a\left( x\right) g\left( \partial _{t}u\left( s,x\right)
\right) \partial _{t}u\left( s,x\right) dxds \\ 
+\left( k-64\left\Vert a\right\Vert _{\infty }\left( C_{T}\right)
^{2}\left\Vert g^{\prime }\right\Vert _{L^{\infty }}\right)
\dint_{t}^{t+T}\varphi ^{\prime }\left( s\right) \dint_{\Omega }a\left(
x\right) g^{\prime }\left( \partial _{t}u\left( s,x\right) \right)
\left\vert \partial _{t}^{2}u\left( s,x\right) \right\vert ^{2}dxds \\ 
\leq C_{1}\left( \dint_{t}^{t+T}\varphi ^{\prime }\left( s\right)
\dint_{\Omega }a\left( x\right) \left( \left\vert \partial _{t}u\left(
s,x\right) \right\vert ^{2}+\left\vert g\left( \partial _{t}u\left(
s,x\right) \right) \right\vert ^{2}\right) dxds\right) \\ 
+X\left( t\right) +\left( 8C_{T}+k+k_{0}\right) \left(
\sum_{i=0}^{1}E_{\partial _{t}^{i}u,\partial _{t}^{i}v}\left( 0\right)
\right) \dint_{t}^{t+T}\left\vert \varphi ^{\prime \prime }\left( s\right)
\right\vert ds.%
\end{array}
\label{proof theorem Xt final estimate1}
\end{equation}

where%
\begin{equation}
C_{1}=C\left( T,\left\Vert b\right\Vert _{\infty },\left\Vert a\right\Vert
_{\infty }\right) .  \label{C1 definition}
\end{equation}

Now we have to estimate the first term of RHS of the estimate above by the
third term of the LHS. We set for all fixed $s\geq 0,~\Omega ^{s}=\left\{
x,\left\vert \partial _{t}u\left( s,x\right) \right\vert <1\right\} .$
Thanks to $\left( \ref{h0 defin}\right) ,$ we have 
\begin{equation*}
\begin{array}{c}
\dint_{t}^{t+T}\varphi ^{\prime }\left( s\right) \dint_{\Omega ^{s}}a\left(
x\right) \left( \left\vert \partial _{t}u\left( s,x\right) \right\vert
^{2}+\left\vert g\left( \partial _{t}u\left( s,x\right) \right) \right\vert
^{2}\right) dxds \\ 
\leq \frac{1}{\epsilon _{0}}\dint_{t}^{t+T}\varphi ^{\prime }\left( s\right)
\dint_{\Omega ^{s}}h_{0}\left( g\left( \partial _{t}u\left( s,x\right)
\right) \partial _{t}u\left( s,x\right) \right) a\left( x\right) dxds \\ 
\leq \frac{1}{\epsilon _{0}}\dint_{t}^{t+T}\varphi ^{\prime }\left( s\right)
\dint_{\Omega }h_{0}\left( g\left( \partial _{t}u\left( s,x\right) \right)
\partial _{t}u\left( s,x\right) \right) a\left( x\right) dxds%
\end{array}%
\end{equation*}%
Since $h_{0}$ is concave, we can use (the reverse) Jensen's inequality and
we obtain 
\begin{equation*}
\begin{array}{c}
\dint_{t}^{t+T}\varphi ^{\prime }\left( s\right) \dint_{\Omega ^{s}}a\left(
x\right) \left( \left\vert \partial _{t}u\left( s,x\right) \right\vert
^{2}+\left\vert g\left( \partial _{t}u\left( s,x\right) \right) \right\vert
^{2}\right) dxds \\ 
\leq \frac{1}{\epsilon _{0}}\dint_{t}^{t+T}\varphi ^{\prime }\left( s\right)
h\left( \dint_{\Omega }\left( g\left( \partial _{t}u\left( s,x\right)
\right) \partial _{t}u\left( s,x\right) \right) a\left( x\right) dx\right)
ds.%
\end{array}%
\end{equation*}%
Under our assumptions the function $H$ \ defined by $\left( \ref{H
definition}\right) $\ is convex and proper. Hence, we can apply Young's
inequality \cite{rockfellar}%
\begin{equation*}
\begin{array}{l}
\dint_{t}^{t+T}\varphi ^{\prime }\left( s\right) \dint_{\Omega ^{s}}a\left(
x\right) \left( \left\vert \partial _{t}u\left( s,x\right) \right\vert
^{2}+\left\vert g\left( \partial _{t}u\left( s,x\right) \right) \right\vert
^{2}\right) dxds \\ 
\leq \left( \frac{\epsilon _{0}}{2C_{1}}\dint_{t}^{t+T}\varphi \left(
s\right) H^{\ast }\left( \frac{2C_{1}\varphi ^{\prime }\left( s\right) }{%
\epsilon _{0}\varphi \left( s\right) }\right) ds+\frac{\epsilon _{0}}{2C_{1}}%
\dint_{t}^{t+T}\varphi \left( s\right) \dint_{\Omega }a\left( x\right)
g\left( \partial _{t}u\left( s,x\right) \right) \partial _{t}u\left(
s,x\right) dxds\right) ,%
\end{array}%
\end{equation*}%
where $H^{\ast }$ is the convex conjugate of the function $H.$

On the other hand, using the fact that the function $g$ is linearly bounded
near infinity, we infer that%
\begin{equation*}
\begin{array}{l}
\dint_{t}^{t+T}\varphi ^{\prime }\left( s\right) \dint_{\Omega }a\left(
x\right) \left( \left\vert \partial _{t}u\left( s,x\right) \right\vert
^{2}+\left\vert g\left( \partial _{t}u\left( s,x\right) \right) \right\vert
^{2}\right) dxds \\ 
\leq \frac{1}{2C_{1}}\dint_{t}^{t+T}\varphi \left( s\right) H^{\ast }\left( 
\frac{2C_{1}\varphi ^{\prime }\left( s\right) }{\epsilon _{0}\varphi \left(
s\right) }\right) ds+\frac{1}{2C_{1}}\dint_{t}^{t+T}\varphi \left( s\right)
\dint_{\Omega }a\left( x\right) g\left( \partial _{t}u\left( s,x\right)
\right) \partial _{t}u\left( s,x\right) dxds \\ 
+\left( \frac{1}{m}+M^{2}\right) \dint_{t}^{t+T}\varphi ^{\prime }\left(
s\right) \dint_{\Omega }a\left( x\right) g\left( \partial _{t}u\left(
s,x\right) \right) \partial _{t}u\left( s,x\right) dxds.%
\end{array}%
\end{equation*}%
The estimate above combined with $\left( \ref{proof theorem Xt final
estimate1}\right) $, give%
\begin{equation}
\begin{array}{l}
X\left( t+T\right) +\frac{1}{2}\dint_{t}^{t+T}\varphi ^{\prime }\left(
s\right) E_{u,v}\left( s\right) ds \\ 
+\dint_{t}^{t+T}\left( \frac{1}{2}-\left( \frac{1}{m}+M^{2}\right) C_{1}%
\frac{\varphi ^{\prime }\left( s\right) }{\varphi \left( s\right) }\right)
\varphi \left( s\right) \dint_{\Omega }a\left( x\right) g\left( \partial
_{t}u\left( s,x\right) \right) \partial _{t}u\left( s,x\right) dxds \\ 
+\left( k-64\left\Vert a\right\Vert _{\infty }\left( C_{T}\right)
^{2}\left\Vert g^{\prime }\right\Vert _{L^{\infty }}\right)
\dint_{t}^{t+T}\varphi ^{\prime }\left( s\right) \dint_{\Omega }a\left(
x\right) g^{\prime }\left( \partial _{t}u\left( s,x\right) \right)
\left\vert \partial _{t}^{2}u\left( s,x\right) \right\vert ^{2}dxds \\ 
\leq X\left( t\right) +\left( 8C_{T}+k+k_{0}\right) \left(
\sum_{i=0}^{1}E_{\partial _{t}^{i}u,\partial _{t}^{i}v}\left( 0\right)
\right) \dint_{t}^{t+T}\left\vert \varphi ^{\prime \prime }\left( s\right)
\right\vert ds+\frac{1}{2}\dint_{t}^{t+T}\varphi \left( s\right) H^{\ast
}\left( \frac{2C_{1}\varphi ^{\prime }\left( s\right) }{\epsilon _{0}\varphi
\left( s\right) }\right) ds.%
\end{array}
\label{proof theorem Xt final estimate}
\end{equation}%
We remind that%
\begin{equation*}
\begin{array}{l}
X\left( t\right) =\varphi ^{\prime }\left( t\right) \dint_{\Omega }u\left(
t,x\right) \partial _{t}u\left( t,x\right) +v\left( t,x\right) \partial
_{t}v\left( t,x\right) dx \\ 
+k_{1}\varphi ^{\prime }\left( t\right) \dint_{\Omega }\partial
_{t}^{2}u\left( t,x\right) \partial _{t}v\left( t,x\right) -\partial
_{t}^{2}v\left( t,x\right) \partial _{t}u\left( t,x\right) dx+\varphi \left(
t\right) E_{u,v}\left( t\right) +k\varphi ^{\prime }\left( t\right)
E_{\partial _{t}u,\partial _{t}v}\left( t\right) ,%
\end{array}%
\end{equation*}%
with $k_{1}=8C_{T}.$ Using Young's inequality and Poincar\'{e} inequality,
it is easy to see that%
\begin{equation}
\begin{array}{l}
X\left( t\right) \leq \left( \varphi \left( t\right) +\frac{\varphi ^{\prime
}\left( t\right) }{\delta }\left( k_{1}+2\lambda ^{2}+1\right) \right)
E_{u,v}\left( t\right) +\varphi ^{\prime }\left( t\right) \left( k+\frac{%
k_{1}}{\delta }\right) E_{\partial _{t}u,\partial _{t}v}\left( t\right) \\ 
X\left( t\right) \geq \left( 1-\frac{\varphi ^{\prime }\left( t\right) }{%
\delta \varphi \left( t\right) }\left( k_{1}+2\lambda ^{2}+1\right) \right)
\varphi \left( t\right) E_{u,v}\left( t\right) +\varphi ^{\prime }\left(
t\right) \left( k-\frac{k_{1}}{\delta }\right) E_{\partial _{t}u,\partial
_{t}v}\left( t\right)%
\end{array}
\label{k estimate}
\end{equation}%
So, taking $k$ such that 
\begin{equation*}
k\geq \max \left( \frac{8C_{T}}{\delta },64\left\Vert a\right\Vert _{\infty
}\left( C_{T}\right) ^{2}\left\Vert g^{\prime }\right\Vert _{L^{\infty
}}\right) ,
\end{equation*}
using $\left( \ref{proof theorem Xt final estimate}\right) $\ and the fact
that%
\begin{equation*}
\begin{array}{c}
\left( \frac{1}{2}-C_{1}\left( \frac{1}{m}+M^{2}\right) \frac{\varphi
^{\prime }\left( t\right) }{\varphi \left( t\right) }\right) \geq 0\text{
and }1-\frac{\varphi ^{\prime }\left( t\right) }{\delta \varphi \left(
t\right) }\left( 8C_{T}+2\lambda ^{2}+1\right) \geq 0,\text{ for all }t\geq
0,%
\end{array}%
\end{equation*}%
we deduce that $X\left( t\right) \geq 0$ and 
\begin{equation}
\begin{array}{l}
X\left( t+T\right) +\frac{1}{2}\dint_{t}^{t+T}\varphi ^{\prime }\left(
s\right) E_{u,v}\left( s\right) ds \\ 
\leq X\left( t\right) +8C_{T}\left( \sum_{i=0}^{1}E_{\partial
_{t}^{i}u,\partial _{t}^{i}v}\left( 0\right) \right)
\dint_{t}^{t+T}\left\vert \varphi ^{\prime \prime }\left( s\right)
\right\vert ds+\frac{1}{2}\dint_{t}^{t+T}\varphi \left( s\right) H^{\ast
}\left( \frac{2C_{1}\varphi ^{\prime }\left( s\right) }{\epsilon _{0}\varphi
\left( s\right) }\right) ds.%
\end{array}
\label{Xt final estimate}
\end{equation}%
for all $t\geq 0$. Thus%
\begin{equation*}
\begin{array}{l}
\overset{n-1}{\underset{i=0}{\sum }}\left( X\left( \left( i+1\right)
T\right) -X\left( iT\right) +\frac{1}{2}\dint_{iT}^{\left( i+1\right)
T}\varphi ^{\prime }\left( s\right) E_{u,v}\left( s\right) ds\right) \\ 
\leq \left( 8C_{T}+k+k_{0}\right) \left( \sum_{i=0}^{1}E_{\partial
_{t}^{i}u,\partial _{t}^{i}v}\left( 0\right) \right)
\dint_{0}^{nT}\left\vert \varphi ^{\prime \prime }\left( s\right)
\right\vert ds+\frac{1}{2}\dint_{0}^{nT}\varphi \left( s\right) H^{\ast
}\left( \frac{2C_{1}\varphi ^{\prime }\left( s\right) }{\epsilon _{0}\varphi
\left( s\right) }\right) ds,%
\end{array}%
\end{equation*}%
and this gives%
\begin{equation*}
\begin{array}{l}
X\left( nT\right) +\frac{1}{2}\dint_{0}^{nT}\varphi ^{\prime }\left(
s\right) E_{u,v}\left( s\right) ds \\ 
\leq X\left( 0\right) +\left( 8C_{T}+k+k_{0}\right) \left(
\sum_{i=0}^{1}E_{\partial _{t}^{i}u,\partial _{t}^{i}v}\left( 0\right)
\right) \dint_{0}^{nT}\left\vert \varphi "\left( s\right) \right\vert ds \\ 
+\frac{1}{2}\dint_{0}^{nT}\varphi \left( s\right) H^{\ast }\left( \frac{%
2C_{1}\varphi ^{\prime }\left( s\right) }{\epsilon _{0}\varphi \left(
s\right) }\right) ds,\text{ for all }n\in 
\mathbb{N}
.%
\end{array}%
\end{equation*}%
Utilizing lemma \ref{proposition phi}, we conclude that there exists a
positive constant $C$ such that 
\begin{equation}
\begin{array}{c}
\dint_{0}^{+\infty }\varphi ^{\prime }\left( s\right) E_{u,v}\left( s\right)
ds\leq C\left( 1+\sum_{i=0}^{1}E_{\partial _{t}^{i}u,\partial
_{t}^{i}v}\left( 0\right) \right) .%
\end{array}
\label{integral bound}
\end{equation}%
Since%
\begin{equation*}
\varphi \left( t\right) E_{u,v}\left( t\right) \leq \varphi \left( 0\right)
E_{u,v}\left( 0\right) +\int_{0}^{+\infty }\varphi ^{\prime }\left( s\right)
E_{u,v}\left( s\right) ds,~t\geq 0,
\end{equation*}%
then $\left( \ref{integral bound}\right) ,$ gives $\left( \ref{energy decay
rate}\right) .$

Finally, using the density of $D\left( \mathcal{A}\right) $ in $\mathcal{H}$%
, we obtain $\left( \ref{limit energy space}\right) .$

\end{document}